\documentclass[11pt]{article}

\usepackage{amsmath, amssymb, amsthm, amsfonts}
\usepackage{geometry}
\usepackage{mathtools}
\usepackage[colorlinks=true,linkcolor=blue,citecolor=teal,urlcolor=black]{hyperref}
\usepackage{outlines}
\usepackage{geometry}
\usepackage{subcaption}
\usepackage{xcolor}
\usepackage{float}

\title{Strengthened upper bound on the third eigenvalue of graphs}
\author{Sida Li\thanks{Department of Pure Mathematics and Mathematical Statistics
(DPMMS), University of Cambridge, United Kingdom. Email: \href{mailto:sl2190@cam.ac.uk}{sl2190@cam.ac.uk}.}}
\date{}

\newcommand{\purp}[1]{\textcolor{purple}{#1}}

\newtheorem{thm}{Theorem}[section]
\newtheorem{dfn}[thm]{Definition}

\newtheorem{clm}[thm]{Claim}
\newtheorem{cor}[thm]{Corollary}
\newtheorem{cnj}[thm]{Conjecture}

\begin{document}
\maketitle

\begin{abstract}
    Let $G$ be a graph on $n \ge 3$ vertices, whose adjacency matrix has eigenvalues $\lambda_1 \ge \lambda_2 \ge \dots \ge \lambda_n$. The problem of bounding $\lambda_k$ in terms of $n$ was first proposed by Hong and was studied by Nikiforov, who demonstrated strong upper and lower bounds for arbitrary $k$. Nikiforov also claimed a strengthened upper bound for $k \ge 3$, namely that $\frac{\lambda_k}{n} < \frac{1}{2\sqrt{k-1}} - \varepsilon_k$ for some positive $\varepsilon_k$, but omitted the proof due to its length. In this paper, we give a proof of this bound for $k = 3$. We achieve this by instead looking at $\lambda_{n-1} + \lambda_n$ and introducing a new graph operation which provides structure to minimising graphs, including $\omega \le 3$ and $\chi \le 4$. Then we reduce the hypothetical worst case to a graph that is $n/2$-regular and invariant under said operation. By considering a series of inequalities on the restricted eigenvector components, we prove that a sequence of graphs with $\frac{\lambda_{n-1} + \lambda_n}{n}$ converging to $-\frac{\sqrt{2}}{2}$ cannot exist. 
\end{abstract}

\section{Introduction}

For a graph $G$ on $n \ge 3$ vertices, let the eigenvalues of the adjacency matrix of $G$ be $\lambda_1 \ge \lambda_2 \ge \dots \ge \lambda_n$. Denote the complement of $G$ by $\overline{G}$, with eigenvalues $\overline{\lambda_1} \ge \dots \ge \overline{\lambda_n}$. 

The problem of bounding the $k$-th eigenvalue of a graph $G$ of order $n$ was first proposed by Hong \cite{hong1993bounds}. In accordance with Nikiforov \cite{nikiforov2015extrema}, let $c_k = \sup\{\lambda_k(G)/n : |V(G)| = n \ge k\}$ and $c_{-k} = \sup\{|\lambda_{n-k+1}(G)|/n : |V(G)| = n \ge k\}$. Nikiforov proved that for all $k \ge 2$,
\[c_k \le \frac{1}{2\sqrt{k-1}}.\]
A natural lower bound is provided by the construction $kK_{\frac{n}{k}}$ which gives $\lambda_k = \frac{n}{k} - 1$, hence $c_k \ge \frac{1}{k}$. For $k = 1, 2$, we indeed have $c_1 = 1$ and $c_2 = \frac12$, but Nikiforov showed that for $k \ge 5$, we have an improvement $c_k \ge \frac{1}{k - 1/2}$. Recently Linz \cite{linz2023improved} demonstrated a construction for $c_4 \ge \frac{1+\sqrt{5}}{12} \approx 0.26967 > \frac14$. 

Nikiforov claimed that for each $k \ge 3$, there exists $\varepsilon_k > 0$ such that
\[c_k < \frac{1}{2\sqrt{k-1}} - \varepsilon_k,\]
though the proof was never published. In this paper, we prove the claim for $k = 3$. 

\begin{thm}\label{thm:worst}
    There exists $\varepsilon_3 > 0$ such that
    \[c_3 < \frac{1}{2\sqrt{2}} - \varepsilon_3.\]
\end{thm}

In comparison to Nikiforov's proof outline, our proof doesn't rely on the Removal Lemma of Alon et al. \cite{alon2000testing} nor analytic graph theory. Nonetheless, the final step is motivated by considering a sequence of worst-case graphs and observing averages that do converge, since graph convergence is not guaranteed. We also provide possible alternative proof ideas that delve into the theory of graphons. 

Instead, we shift our focus to the related but untouched problem of minimising $\lambda_{n-1} + \lambda_n$ over all graphs of order $n$. This is motivated by observing the inequalities used to derive the original $c_3 \le \frac{1}{2\sqrt{2}}$ bound by Nikiforov \cite{nikiforov2015extrema} and determining conditions for which they are tight. Indeed, we first have $\lambda_1 = \frac{2e(G)}{n}$ which corresponds to regular $G$, and in the last step we require $\lambda_2 = \lambda_3$, which Leonida and Li \cite{leonida2024structure} also provided strong evidence for. In this case, minimal $\overline{\lambda_{n-1}} + \overline{\lambda_n}$ is indeed equivalent to maximal $\lambda_3$. 

We prove the following theorem, which strengthens the current bound of $\lambda_{n-1} + \lambda_n \ge -\frac{\sqrt{2}}{2}n$, given by AM-QM on $\lambda_{n-1}^2 + \lambda_n^2 \le \frac{n^2}{4}$. The infimum exists by Nikiforov \cite{nikiforov2006combinations}. 

\begin{thm}\label{thm:worst-sum}
    There exists $\varepsilon > 0$ such that
    \[\inf\left\{\frac{\lambda_{n-1} + \lambda_n}{n} : |V(G)| = n \ge 3\right\} > -\frac{\sqrt{2}}{2} + \varepsilon.\]
\end{thm}

Theorem \ref{thm:worst-sum} and $\lambda_{n-1} \ge \lambda_n$ immediately give:

\begin{cor}\label{cor:worst-2}
    There exists $\varepsilon_3 > 0$ such that
    \[c_{-2} < \frac{1}{2\sqrt{2}} - \varepsilon_3.\]
\end{cor}

Then Theorem \ref{thm:worst} follows immediately from Weyl's inequalities, which gives $\lambda_3 + \overline{\lambda_{n-1}} \le \lambda_2(K_n) = -1$. \\

The structure of the rest of the paper, dedicated to proving Theorem \ref{thm:worst-sum}, is as follows. In Section \ref{section:structure}, we first introduce a new graph operation which restricts the structure of graphs minimising $\lambda_{n-1} + \lambda_n$, forcing them to look not dissimilar to $Pi_n$ as introduced by Leonida and Li \cite{leonida2024structure}. For example, we prove the following, where $\omega$ and $\chi$ denote the clique and chromatic numbers respectively. 

\begin{thm}\label{thm:worst-structure-intro}
    For every $n \ge 3$, there exists $G$ of order $n$ with minimal $\lambda_{n-1} + \lambda_n$ such that $\omega(G) \le 3$ and $\chi(G) \le 4$. 
\end{thm}

This operation was inspired by the analogous operation for $\lambda_1 + \lambda_2$, proposed by Ebrahimi et al. \cite{ebrahimi2008tau2}. We then briefly describe a generalisation, though not involved in the proof. 

Then in Section \ref{section:n/2-regular} we further impose the graph to be $n/2$-regular and prove the theorem for this family of constrained invariant graphs. We do this by deriving a series of inequalities from eigenvector components, where we consider two different patterns of eigenvectors that each impose distinct structural restrictions on the graph. Then, given a sequence of graphs with $\frac{\lambda_{n-1} + \lambda_n}{n}$ converging to $-\frac{\sqrt{2}}{2}$, we take local averages of degrees and eigenvector components to prove that these simultaneous restrictions yield a contradiction.

We conclude in Section \ref{section:proof} by reducing the hypothetical `worst-case' graph to this family by showing that we can obtain the same required structure with $o(n^2)$ edge additions/removals, which doesn't affect the overall growth of the spectrum.

\section{Structure of minimising graphs}\label{section:structure}

We introduce a new operation similar to the one used to bound $\lambda_1 + \lambda_2$ by Ebrahimi et al. \cite{ebrahimi2008tau2}. Since we are focussing on the last two eigenvalues, intuitively this operation ought to increase the third eigenvalue proportion of $\overline{G}$. 

\begin{dfn}\label{dfn:operation}
    Define \purp{Operation $^*$} on a graph $G$ as follows. Let $G^*$ be any graph such that $V(G^*) = V(G)$ and there exists an orthonormal pair of eigenvectors $\{x, y\}$ corresponding to the eigenvalues $\{\lambda_{n-1}(G), \lambda_n(G)\}$ of $G$ such that in $G^*$, we have $i \sim j \iff x_ix_j + y_iy_j < 0$. 
\end{dfn}

The motivation for this operation is as follows.

\begin{thm}\label{thm:operation-monovar}
    For any $G^*$, we have $\lambda_{n-1}(G^*) + \lambda_n(G^*) \le \lambda_{n-1}(G) + \lambda_n(G)$. 
\end{thm}
\begin{proof}
    As $A(G), A(G^*)$ are Hermitian matrices, we have:
    \begin{align*}
        \lambda_{n-1}(G^*) + \lambda_n(G^*) &= \inf_{\substack{|z| = |w| = 1\\z \perp w}} \langle z, A(G^*)z\rangle + \langle w, A(G^*)w \rangle \\
        &\le \langle x, A(G^*)x\rangle + \langle y, A(G^*)y \rangle \\
        &= 2 \cdot \sum_{x_ix_j + y_iy_j < 0} (x_ix_j + y_iy_j) \\
        &\le 2 \cdot \sum_{ij \in E(G)} (x_ix_j + y_iy_j) \\
        &= \langle x, A(G)x\rangle + \langle y, A(G)y \rangle = \lambda_{n-1}(G) + \lambda_n(G). \qedhere
    \end{align*}
\end{proof}

\begin{cor}\label{cor:worst-invariant}
    For every $n \ge 3$, there exists $G$ of order $n$ with minimal $\lambda_{n-1} + \lambda_n$ such that $G = G^*$. 
\end{cor}
\begin{proof}
    Since there are finitely many graphs of order $n$, pick $G$ with minimal $\lambda_{n-1} + \lambda_n$ and furthermore with minimal number of edges, then we must have $\lambda_{n-1}(G^*) + \lambda_n(G^*) = \lambda_{n-1}(G) + \lambda_n(G)$. Hence $G^*$ can only possibly differ from $G$ by the removal of edges where $x_ix_j + y_iy_j = 0$. Yet $G$ was selected with minimal number of edges hence no edges were removed, so $G = G^*$. 
\end{proof}

To each vertex of $G$, associate a $\mathbb{R}^2$ vector $v_i = \binom{x_i}{y_i}$. Then the adjacency condition becomes: $i \sim j \iff v_i \cdot v_j < 0 \iff v_i, v_j \text{ are non-zero and } \angle (v_i, v_j) > \pi/2$.

Note that $G^*$ is not necessarily unique if the geometric multiplicity of the last eigenvalue is at least 3 or the geometric multiplicity of the second-last eigenvalue is at least 2, due to the choice in orthonormal vectors. If the geometric multiplicity of the last is exactly 2, then rotating our orthonormal pair corresponds to rotating the $v_i$, hence adjacencies are invariant and $G^*$ is unique. $G^*$ is immediately unique if the geometric multiplicity of the last and second-last eigenvalues are both 1. 

We now demonstrate that the output graphs of this operation have a rich structure, with traces of pivalous graphs as introduced by Leonida and Li \cite{leonida2024structure}. 

\begin{thm}\label{thm:operation-run}
    Suppose $G = H^*$ for some graph $H$. Then there exists a labelling of the vertices $[1, 2, \dots, n]$ such that the neighbours of vertex $i$ are $\{a_i, a_i+1, \dots, b_i\}$ mod $n$, and that the $(a_i), (b_i)$ are increasing. 
\end{thm}
\begin{proof}    
    Plot all the vectors $v_1, \dots, v_n$ and re-label them in anti-clockwise order, with WLOG (without loss of generality) all zero vectors first. Since the neighbours of $i$ are exactly the vectors at an obtuse angle away from $v_i$, or empty if $v_i = 0$, this forms a consecutive run taken mod $n$ where necessary. It also satisfies the fact that the start and end-points of the run increase as we rotate anti-clockwise around the origin.
\end{proof}

\begin{thm}\label{thm:operation-clique}
    Suppose $G = H^*$ for some graph $H$. Then $\omega(G) \le 3$. 
\end{thm}
\begin{proof}
    Suppose FTSOC (for the sake of contradiction) there exists a $K_4$, say non-zero $w_1, w_2, w_3, w_4$ such that $w_i \cdot w_j < 0$ for all pairs $\{i, j\} \subset \{1, 2, 3, 4\}$. Draw the four vectors on the plane. A pair of consecutive vectors must have an acute angle between them, otherwise the sum of the angles is $> 2\pi$. Between these two, $w_i \cdot w_j \ge 0$, contradiction.
\end{proof}

\begin{thm}\label{thm:operation-chromatic}
    Suppose $G = H^*$ for some graph $H$. Then $\chi(G) \le 4$. 
\end{thm}
\begin{proof}
    Consider the partition of the vertices into quadrants of the usual Cartesian plane, i.e. by sign of $x_i$ and $y_i$ with WLOG zero included into the positives. Then vectors within a quadrant differ by an acute or right angle, hence form an independent set.
\end{proof}

Combining Corollary \ref{cor:worst-invariant} with Theorems \ref{thm:operation-clique} and \ref{thm:operation-chromatic} immediately yields Theorem \ref{thm:worst-structure-intro}. We explore some specific sub-cases of output graphs which admit additional structure. 

\begin{thm}\label{thm:operation-three}
    Suppose $G = H^*$ for some graph $H$. If there exist two anti-clockwise consecutive neighbours in $\{v_1, \dots, v_n\}$ which differ by an obtuse angle, then $\chi(G) \le 3$. 
\end{thm}
\begin{proof}
    Rotate the Cartesian plane to force at most three quadrants to contain a vector, with the last quadrant contained inside this obtuse angle gap. Thus, we get at most three independent sets. 

    Alternatively, one can notice that $\overline{G}$ is a chordal graph, hence perfect. By the perfect graph theorem, $G$ is perfect hence has $\chi(G) = \omega(G) \le 3$. Both results are corollaries of the strong perfect graph theorem of Chudnovsky et al. \cite{chudnovsky2006perfect}.
\end{proof}

Recall that a closed blow-up by $k$ of a graph is constructed by replacing each vertex with a $K_k$ and each edge with a complete bipartite graph between the endpoint cliques.

\begin{thm}\label{thm:operation-regular}
    Suppose $G = H^*$ for some graph $H$. If $G$ is regular and each $v_i$ occurs exactly $k$ times for some positive integer $k$, then $\overline{G}$ is the closed blow-up by $k$ of a circulant graph. 
\end{thm}
\begin{proof}
    Consider $\overline{G}$ where $i \sim j \iff v_i \cdot v_j \ge 0$. Note that if $v_i = v_j$, then they have the exact same neighbourhood, hence $\overline{G}$ is a closed blow-up by $k$ of say $G'$. 

    Now suppose $G'$ has regularity $d$. We have $\frac{n}{k}$ distinct vectors now, say $(v_i')$, where $\lambda v_i'$ is the sum of a run of $d$ consecutive vectors. Yet each such run is determined uniquely by the first vector, hence there are $\frac{n}{k}$ such runs and they biject with the vectors. Given Theorem \ref{thm:operation-run}, it must necessarily be circulant. 
\end{proof}

\begin{cor}\label{cor:regular-output}
    Suppose $G = H^*$ for some graph $H$. If $G$ is regular and each $v_i$ occurs exactly $k$ times for some positive integer $k$, then $\frac{\lambda_3(G)}{|V(G)|} \le \frac13$. 
\end{cor}
\begin{proof}
    Leonida and Li \cite{leonida2024structure} proved the result for abelian Cayley graphs, and circulant graphs are Cayley graphs over $\mathbb{Z}_n$. Note that the closed blow-up of a circulant graph is still circulant, hence the result follows from Theorem \ref{thm:operation-regular}. 
\end{proof}

\subsection{Examples of invariant families}

In alignment with the intuition that decreasing $\overline{\lambda_{n-1}} + \overline{\lambda_n}$ ought to correspond to increasing $\lambda_3$, we prove that many families of graphs with large third eigenvalue proportion are invariant under the joint operation of $G \to \overline{G} \to \overline{G}^* \to \overline{\overline{G}^*}$. Alternatively, $\overline{G}$ with significantly negative $\lambda_{n-1}$ proportion is invariant under $^*$. 

Following Leonida and Li \cite{leonida2024structure}, define $H_{a,b}$ to be the closed vertex multiplication of $C_6$ by $[a,b,a,b,a,b]$ and $Pi_n$ the pivalous graph on $n$ vertices. They proved that as $n \to \infty$ (or $3(a+b) \to \infty$), we have $\frac{\lambda_3(H_{a,b})}{3(a+b)} \to \frac13$ and $\frac{\lambda_3(Pi_n)}{n} \to \frac1\pi$. 

\begin{figure}[H]
    \centering
    \begin{tabular}{cc}
         \includegraphics[width = 0.25\textwidth]{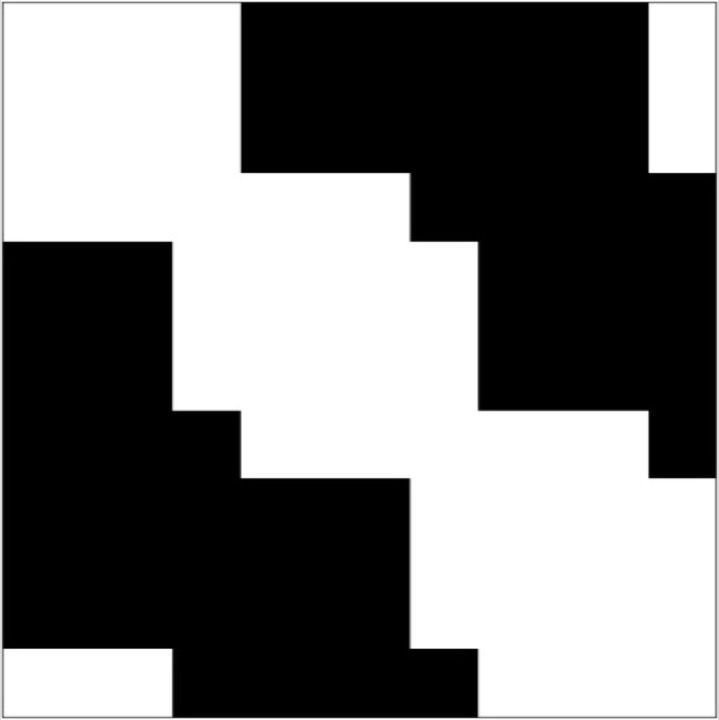} & \includegraphics[width = 0.25\textwidth]{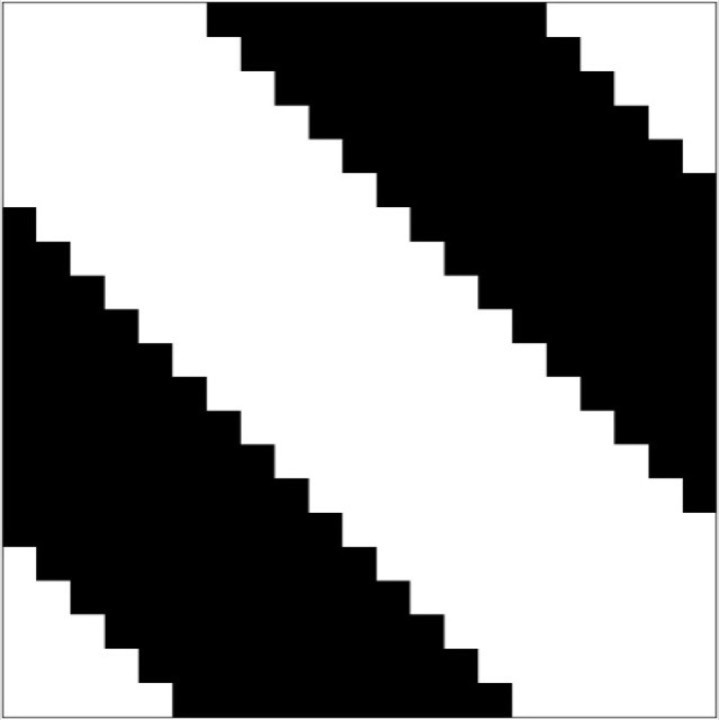} \\
    \end{tabular}
    \caption{Invariant families. $\overline{H_{5,2}}$ on the left, $\overline{Pi_{21}}$ on the right.}
    \label{fig:invariant-families}
\end{figure}

\begin{thm}\label{thm:H-invariant}
    $H_{a,b}$ is invariant under $G \to \overline{G} \to \overline{G}^* \to \overline{\overline{G}^*}$.
\end{thm}
\begin{proof}
    We claim that for $\omega = \exp(\pi i/3)$,
    \[z = \sqrt{\frac{2}{3(a+b)}}(\underbrace{1, \dots, 1}_a, \underbrace{\omega, \dots, \omega}_b, \underbrace{\omega^2, \dots, \omega^2}_a, \underbrace{\omega^3, \dots, \omega^3}_b, \underbrace{\omega^4, \dots, \omega^4}_a, \underbrace{\omega^5, \dots, \omega^5}_b)\]
    is an eigenvector of $\overline{H_{a, b}}$ of eigenvalue $-(a+b)$. Indeed, note that $\omega^k$ is adjacent to $a \cdot \omega^{k+2} + b \cdot \omega^{k+3} + a \cdot \omega^{k+4}$ or vice versa swapping $a, b$. But nonetheless we get $\omega^k \omega^3 (a(\omega + \omega^{-1}) + b) = -\omega^k(a+b)$, hence $z$ is an eigenvector of eigenvalue $-(a+b)$. 
    
    Let the real and imaginary parts of $z$ be $\xi, \eta$. We have $z = \xi + i\eta$, hence $z^T z = (|\xi|^2 - |\eta|^2) + 2i\xi \cdot \eta$ and $z^T \overline{z} = |\xi|^2 + |\eta|^2$. Now $z^T z = \frac{2}{3(a+b)}(a(1 + \omega^4 + \omega^8) + b(\omega^2 + \omega^6 + \omega^{10})) = 0$, hence $|\xi|^2 = |\eta|^2$ and $\xi \cdot \eta = 0$. Then $z^T \overline{z} = \frac{2}{3(a+b)}(a(1 + 1 + 1) + b(1 + 1 + 1)) = 2$. Thus, $\xi, \eta$ are orthonormal, and must be the last two eigenvectors since $|\lambda_{n-2}| \le \frac{n}{2\sqrt{3}} < \frac{n}{3}$. 
    
    We get that the $\mathbb{R}^2$ vectors are the (scaled) $6$th roots of unity in order, clustered with multiplicities $[a,b,a,b,a,b]$, and they indeed already satisfy $i \sim j \iff \angle(v_i, v_j) > \pi/2$.
\end{proof}

\begin{thm}\label{thm:pivalous-invariant}
    $Pi_n$ is invariant under $G \to \overline{G} \to \overline{G}^* \to \overline{\overline{G}^*}$.
\end{thm}
\begin{proof}
    Identically to above, the real and imaginary parts of $\sqrt{\frac{2}{n}}(1, \omega, \dots, \omega^{n-1})$ give the desired last two eigenvectors of $\overline{Pi_n}$ where $\omega = \exp(2\pi i/n)$. The $\mathbb{R}^2$ vectors are therefore the (scaled) $n$-th roots of unity in order and already satisfy $i \sim j \iff \angle(v_i, v_j) > \pi/2$. 
\end{proof}

\subsection{Generalisation to more eigenvalues}\label{sub:generalisation}

Although not immediately relevant to our proof of Theorem \ref{thm:worst-sum}, we quickly note a generalisation of Operation $^*$ to more eigenvalues. 

\begin{dfn}\label{dfn:generalised-operation}
    For positive integer $k$, define \purp{Operation $^{*k}$} on graph $G$ of order $n \ge k$ as follows. Let $G^{*k}$ be any graph such that $V(G^{*k}) = V(G)$ and there exists an orthonormal set of eigenvectors $\{x_{n-k+1}, x_{n-k+2}, \dots, x_n\}$ corresponding to eigenvalues $\{\lambda_{n-k+1}, \dots, \lambda_n\}$ respectively of $G$ such that in $G^{*k}$, we have $i \sim j \iff \sum_{r = 1}^k (x_{n-r+1})_i(x_{n-r+1})_j < 0$. 
\end{dfn}

In particular, $G^* = G^{*2}$. We can prove identically to Theorem \ref{thm:operation-monovar} and Corollary \ref{cor:worst-invariant} the analogous results for $^{*k}$. Indeed, we still have
\begin{align*}
    \lambda_{n-k+1}(G^{*k}) + \dots + \lambda_n(G^{*k}) &= \inf_{U^TU = I} Tr(U^TA(G^{*k})U) \\
    &= \inf_{u_i^Tu_j = \delta_{ij}} \sum_{r = 1}^k \langle u_{n-r+1}, A(G^{*k}) u_{n-r+1} \rangle,
\end{align*}
and as before, we pick $G$ with minimal number of edges.

\begin{thm}\label{thm:generalised-operation-monovar}
    For any $G^{*k}$, we have
    \[\sum_{r = 1}^k \lambda_{n-r+1}(G^{*k}) \le \sum_{r = 1}^k \lambda_{n-r+1}(G).\]
\end{thm}

\begin{cor}\label{cor:generalised-worst-invariant}
    For every $n \ge k$, there exists $G$ of order $n$ with minimal $\lambda_{n-k+1} + \dots + \lambda_n$ such that $G = G^{*k}$.
\end{cor}

Analogous to above, to each vertex associate a $\mathbb{R}^k$ vector
\[v_i = ((x_{n-k+1})_i, (x_{n-k+2})_i, \dots, (x_n)_i)^T.\]

The adjacency condition is again $i \sim j \iff v_i \cdot v_j < 0$. We prove the generalisations of Theorems \ref{thm:operation-clique} and \ref{thm:operation-chromatic}. 

\begin{thm}\label{thm:generalised-operation-clique}
    Suppose $G = H^{*k}$ for some graph $H$. Then $\omega(G) \le k+1$.
\end{thm}
\begin{proof}
    Rankin proved that we cannot have $w_1, \dots, w_{k+2}$ such that $w_i \cdot w_j < 0$ for all pairs in Lemma 8 of \cite{rankin1947packing}. Thus, there is no $K_{k+2}$ sub-graph. 
\end{proof}

\begin{thm}\label{thm:generalised-operation-chromatic}
    Suppose $G = H^{*k}$ for some graph $H$. Then $\chi(G) \le 2^k$.
\end{thm}
\begin{proof}
    Again partition $\mathbb{R}^k$ by sign of each component, then for $v_i, v_j$ in the same part, we have $(v_i)_r(v_j)_r \ge 0$ for all $r = 1, \dots, k$. Thus $v_i \cdot v_j = \sum_{r = 1}^k (v_i)_r(v_j)_r \ge 0$, hence this forms an independent set. There are $2^k$ such parts formed by a choice of two signs for each component. 
\end{proof}

\section{Imposing $n/2$-regularity}\label{section:n/2-regular}

We first analyse $G$ for which $G = G^*$ and $G$ is $n/2$-regular. This represents a hypothetical worst-case graph in the sense of Nikiforov's bounds in \cite{nikiforov2015extrema} and in light of Corollary \ref{cor:worst-invariant}. In this section we prove the following theorem, and in the next section, seek to formalise this intuition by extending it to the worst-case constraints. 

\begin{thm}\label{thm:n/2-regular}
    There exists $\varepsilon > 0$ such that
    \[\inf\left\{\frac{\lambda_{n-1}(G) + \lambda_n(G)}{n} : G = G^*, G \textrm{ is $n/2$-regular}\right\} > -\frac{\sqrt{2}}{2} + \varepsilon.\]
\end{thm}

We simplify the problem to that of a symmetric matrix with entries in $\{-1, 1\}$ with incredibly rich structure. A perfect elimination ordering is an ordering of the vertices such that each vertex forms a clique with its subsequent neighbours in the order.

\begin{thm}\label{thm:invariant-symmetry}
    Suppose $G = G^*$ and $G$ is $n/2$-regular. Then for any rotation of the labelling in Theorem \ref{thm:operation-run}, we have
    \[A(G) = \begin{pmatrix}
        Q & J_{n/2} - Q \\
        J_{n/2} - Q & Q
    \end{pmatrix}\]
    where the complement of $Q$ is in a perfect elimination ordering. 
\end{thm}
\begin{proof}
    Given $k \le n/2$, if vertex $k$ has neighbours $\{j, j+1, \dots, n/2\} \subset \{1, 2, \dots, n/2\}$, then its neighbours in $\{n/2+1, n/2+2, \dots, n\}$ are necessarily $\{n/2+1, \dots, n/2+j-1\}$ by Theorem \ref{thm:operation-run} and the degree being $n/2$. Identically, if $k$ has neighbours $\{1, 2, \dots, j\}$ for $j < k \le n/2$, then its other neighbours must be $\{n/2+j+1, \dots, n\}$. Thus, if the top-left $\frac{n}{2} \times \frac{n}{2}$ sub-matrix is $Q$, the top-right must be $J-Q$. $Q$ is symmetric hence so is $J-Q$, and by $n/2$-regularity we can then fill out the rest of $A(G)$ as shown. 

    By virtue of Theorem \ref{thm:operation-run}, the complement of $Q$ has a perfect elimination ordering. 
\end{proof}

\begin{cor}\label{cor:invariant-spectrum}
    Suppose $G = G^*$ and $G$ is $n/2$-regular. Then for $Q$ constructed as in Theorem \ref{thm:invariant-symmetry}, the spectrum of $G$ is the union of the spectra of $J$ and $2Q-J$. 
\end{cor}
\begin{proof}
    We have
    \[\det(A-\lambda I) = \det \begin{pmatrix}
        Q - \lambda I & J - Q \\
        J - Q & Q - \lambda I
    \end{pmatrix} = \det((Q - \lambda I) + (J - Q)) \det((Q - \lambda I) - (J-Q)),\]
    which reduces to $\det(J - \lambda I)\det((2Q-J)-\lambda I)$. 
\end{proof}

We now fix the choice of rotation for further control over $Q$. 

\begin{thm}\label{thm:invariant-bipartite}
    Suppose $G = G^*$ and $G$ is $n/2$-regular. Then there exists a \purp{canonical} rotation of the labelling such that $Q$ is the adjacency matrix for a graph on $n/2$ vertices which satisfies the following property. The vertices of $Q$ can be partitioned into \purp{front}, non-empty \purp{middle}, \purp{back} where there are only edges between the front and back, and the middle consists of isolated vertices.
\end{thm}
\begin{proof}
    By regularity we have $\sum v_i = \sum \binom{x_i}{y_i} = 0$ since $x,y$ are orthogonal to $j$. Noting that the lower half of $A(G)$ is the complement of the upper half, for $i \le n/2$ we have $\lambda x_i = \sum_{j \sim i} x_j = \sum_{j = 1}^n x_j - \sum_{j \not \sim i} x_j = 0 - \sum_{j \sim (i + n/2)} x_j = -\lambda x_{i + n/2}$, hence $x_i = -x_{i+n/2}$ and identically $y_i = -y_{i+n/2}$, so $v_i = -v_{i+n/2}$. 

    Each vector has a corresponding reflection in the origin, which means that for every angle $\theta$, taking all vectors between $[\theta, \theta+\pi)$ gives exactly $n/2$ vectors. Note that if $v_i = v_j$, then they lie in the same neighbourhoods and non-neighbourhoods. Thus, supposing that the $n$ vectors yield exactly $k$ distinct vectors, there are exactly $k$ possible sets of $n/2$ consecutive vectors that can be non-neighbourhoods, as they are uniquely determined by the first vector going anti-clockwise. 

    Then the $n/2$ vectors in $[0, \pi)$ are the non-neighbours of some vector $v_{k+1}$ in the top half, and we rotate to force these to the top $n/2$ vertices of $G$. Let the middle consist of all such isolated vectors in this top half, necessarily non-empty (since $v_{k+1}$ exists) and all equal. Then going anti-clockwise in $[0, \pi)$, refer to the front as the vectors before these and the back the vectors after. Vectors in the front lie on the same side of $v_{k+1}$ hence are pairwise non-neighbours, similarly with the back.
\end{proof}

\begin{cor}\label{cor:invariant-degrees}
    Suppose $G = G^*$ and $G$ is $n/2$-regular. Let $Q$ be constructed from its canonical rotation, then the degrees of vertices in the front are decreasing, and in the back are increasing.
\end{cor}
\begin{proof}
    Follows from the complement having a perfect elimination ordering in Theorem \ref{thm:invariant-symmetry} and the bipartition in Theorem \ref{thm:invariant-bipartite}.
\end{proof}

In addition to restrictions on the structure of the graph, we also find order within the eigenvector components, which we split into two types. 

\begin{thm}\label{thm:2Q-J}
    Suppose $G = G^*$ and $G$ is $n/2$-regular. Let $Q$ be constructed from its canonical rotation, then the last two eigenvectors of $2Q-J$ are of two different types:
    \begin{outline}[enumerate]
        \1 \purp{Type 1}: The front entries are increasing up to the middle and decreasing after, with all entries non-negative. 
        \1 \purp{Type 2}: The entries are always decreasing, and positive in the front, non-positive in the back. 
    \end{outline}
\end{thm}
\begin{proof}
    For $i \le n/2$, 
    \[\lambda_{n-1} x_i = \sum_{j \sim i} x_j = \sum_{\substack{j \sim i \\ j \le n/2}} x_j + \sum_{\substack{j \sim i \\ j > n/2}} x_j = \sum_{\substack{j \sim i \\ j \le n/2}} x_j - \sum_{\substack{j \not \sim i \\ j \le n/2}} x_j = \sum_{j = 1}^{n/2} (2Q-J)_{ij} x_j.\]
    
    In particular, the top half of entries of $x,y$ are still eigenvectors of $2Q-J$ with corresponding eigenvalues. 

    Let $v_{k+1}$ lie in the middle. We claim that $x$ is Type 2 and $y$ is Type 1. First we prove the signs. As all vectors have argument $[0, \pi)$ we have all $y_i$ non-negative. Suppose FTSOC there exists a back vector with $x_i$ positive. Then we claim the neighbourhoods of $v_i$ and $v_{k+1}$ are the same, contradiction. Indeed, both are non-adjacent to $v_1$ since they lie in the same quadrant of the Cartesian plane. But $v_{k+1}$ is non-adjacent to $v_n$ and $v_i$ is anti-clockwise of $v_{k+1}$, hence $v_i$ is also non-adjacent to $v_n$ and this uniquely determines both their neighbourhoods. Thus, $x_i$ is non-positive. Analogously, for front vectors we get $x_i$ positive, where zero can't occur since the $\pi$ endpoint is open.

    Now we consider the piecewise monotonicity. Consider going anti-clockwise from a vector $v_i$ to the next vector $v_{i+1}$ in the original space of $n$ vectors. If the neighbourhood changes, then by the bijection of distinct $n/2$ runs to distinct vectors, the neighbourhood changes by the addition of all vectors equal to some $v_j$ and the removal of all vectors equal to $v_{j+n/2} = -v_j$. If two distinct vectors are included or excluded, then these will have the same neighbourhood, contradiction. 

    In particular, we get that $\lambda_{n-1}x_{i+1} - \lambda_{n-1}x_i = 2mx_j$ and $\lambda_ny_{i+1} - \lambda_ny_i = 2my_j$ for some positive integer $m$, and $v_j$ a vector at most $\pi/2$ radians clockwise from $v_i$ but at least $\pi/2$ radians clockwise from $v_{i+1}$. For a front vector $v_i$, since the middle is non-adjacent to all vectors with argument $[0, \pi)$, we must then have $x_j \ge 0$ and $y_j \le 0$ whence $x_{i+1} - x_i \le 0$ and $y_{i+1} - y_i \ge 0$. For the final middle vector and all but the final back vector $v_i$, we must have $v_j$ a front vector, hence $x_j \ge 0, y_j \ge 0$ which implies $x_{i+1} - x_i \le 0, y_{i+1} - y_i \le 0$. 
\end{proof}

To illustrate the structure, we provide two examples. First, consider the following 4-regular $G_4 = G_4^*$ on 8 vertices. We call the graph a \purp{block construction} since $Q$ (constructed from its canonical rotation) is complete bipartite between the front and back.  

\begin{figure}[H]
    \centering
    \begin{tabular}{cccc}
        \includegraphics[width = 0.23\textwidth]{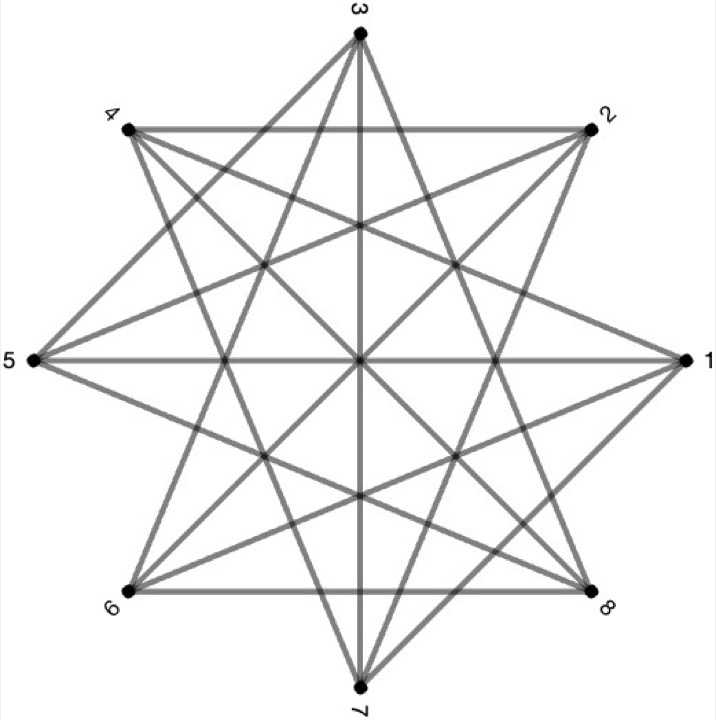} & \includegraphics[width = 0.23\textwidth]{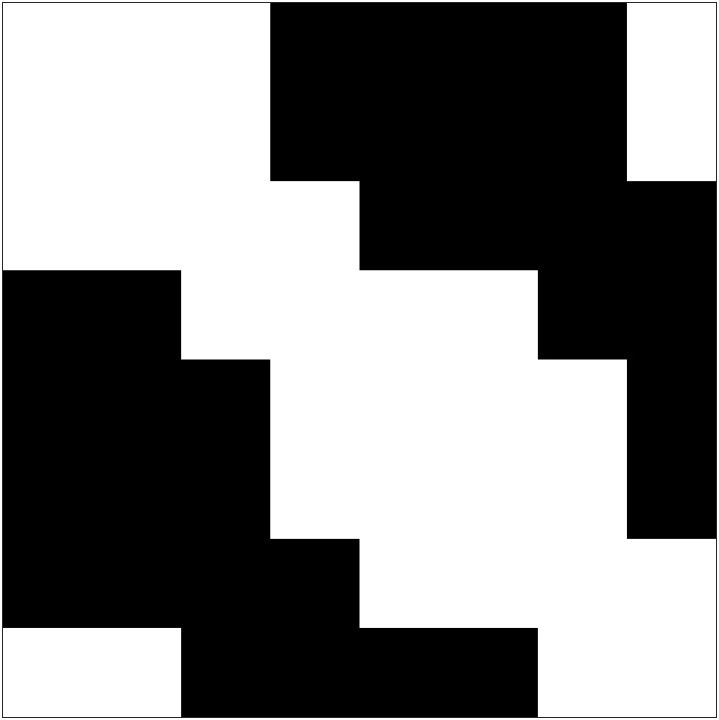} & \includegraphics[width = 0.23\textwidth]{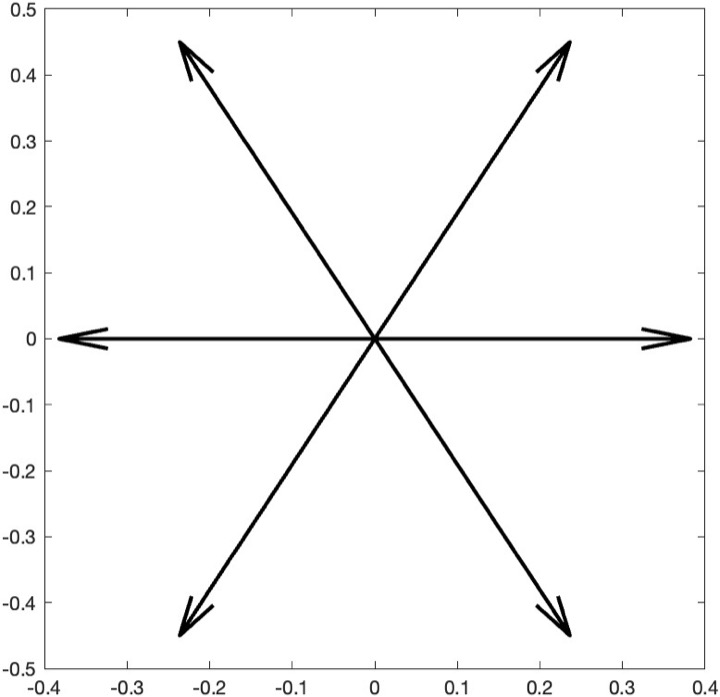} & \includegraphics[width = 0.23\textwidth]{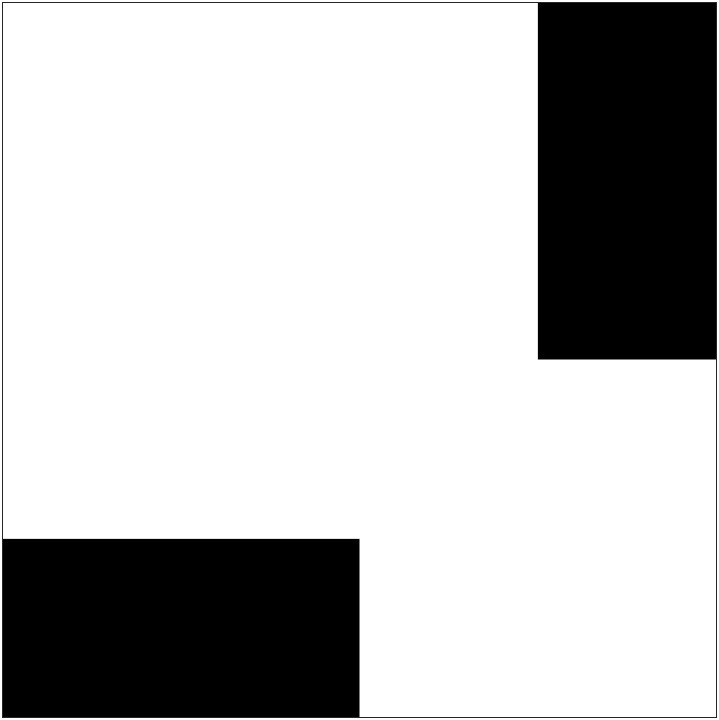} \\
    \end{tabular}
    \caption{From left to right: $G_4$, $A(G_4)$, the resultant $\mathbb{R}^2$ vectors, $Q$.}
    \label{fig:G4}
\end{figure}

In anti-clockwise order, the vectors we obtain are approximately:
\[v_1 = (0.43, 0), \; v_2 = (0.43, 0), \; v_3 = (0.26, 0.5), \; v_4 = (-0.26, 0.5),\]
\[v_5 = (-0.43, 0), \; v_6 = (-0.43, 0), \; v_7 = (-0.26, -0.5), \; v_8 = (0.26, -0.5),\]
hence the last two eigenvectors of $2Q-J$ are $x = (0.43, 0.43, 0.26, -0.26)$ of Type 2 and $y = (0, 0, 0.5, 0.5)$ of Type 1. These correspond to eigenvalues $\approx -3.2361$ and $-2$ respectively, which yields $\frac{\lambda_{n-1} + \lambda_n}{n} \approx -0.6545$. Furthermore, note that the front consists of $\{v_1, v_2\}$, the middle is $\{v_3\}$ and the back is $\{v_4\}$. 

We also consider a non-block 6-regular $G_6 = G_6^*$ on 12 vertices. In anti-clockwise order, the vectors obtained are approximately:
\[v_1 = (0.37, 0), \; v_2 = (0.37, 0), \; v_3 = (0.30, 0.26), \; v_4 = (0.16, 0.43),\]
\[v_5 = (-0.16, 0.43), \; v_6 = (-0.30, 0.26), \; v_7 = (-0.37, 0), \; v_8 = (-0.37, 0),\]
\[v_9 = (-0.30, -0.26), \; v_{10} = (-0.16, -0.43), \; v_{11} = (0.16, -0.43), \; v_{12} = (0.30, -0.26).\]

\begin{figure}[H]
    \centering
    \begin{tabular}{cccc}
        \includegraphics[width = 0.23\textwidth]{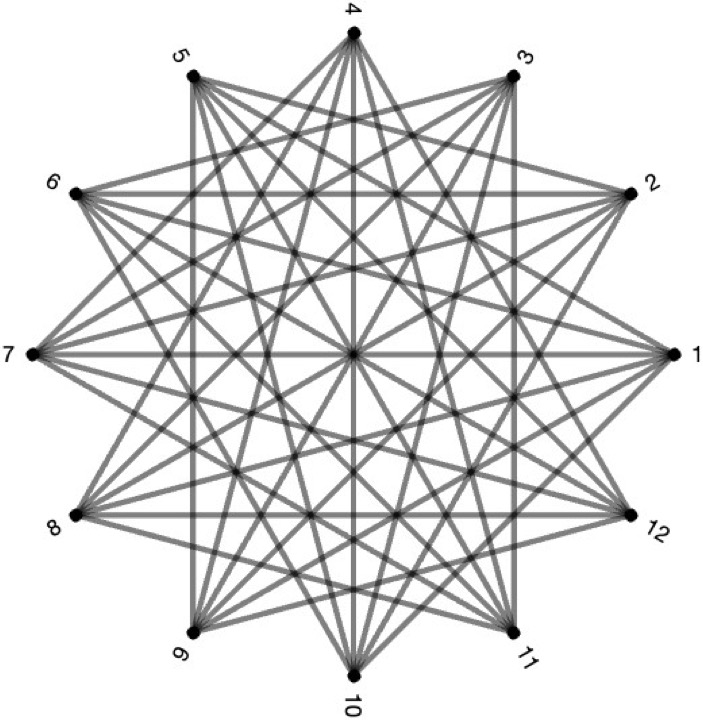} & \includegraphics[width = 0.23\textwidth]{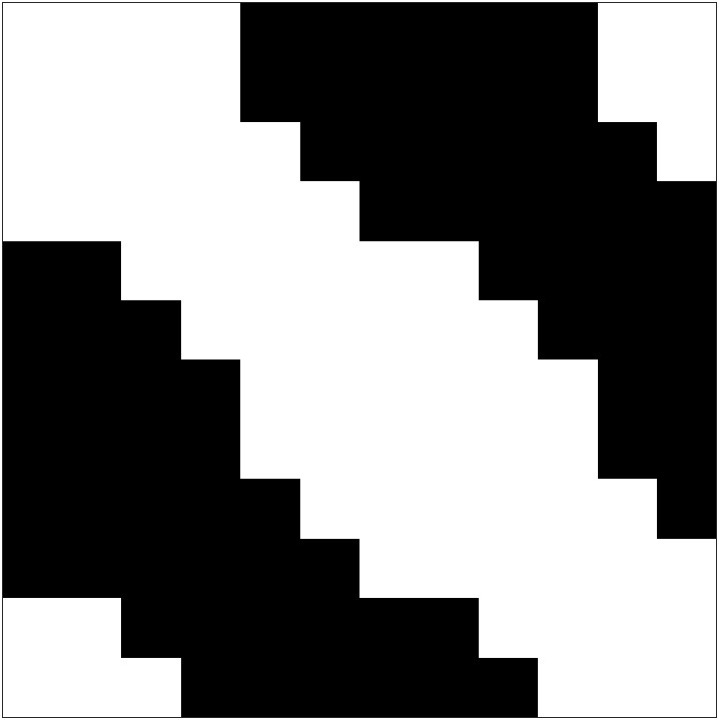} & \includegraphics[width = 0.23\textwidth]{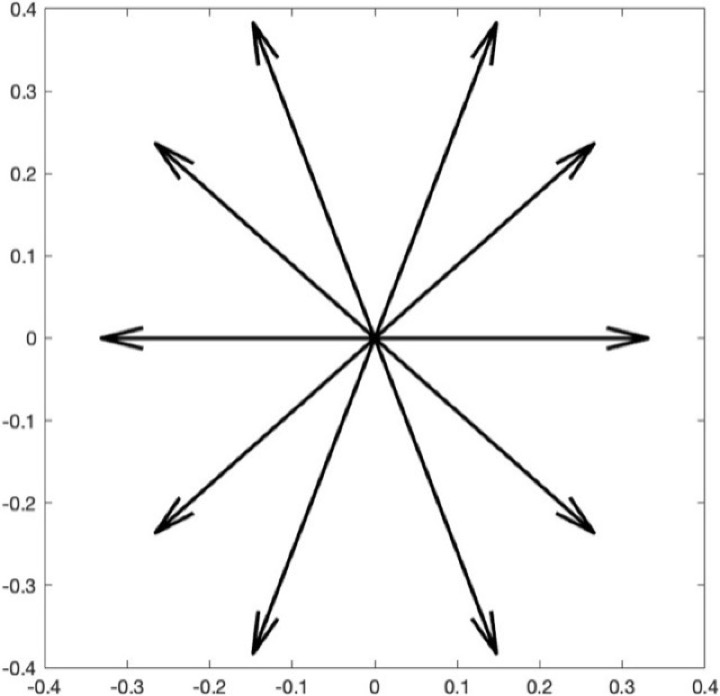} & \includegraphics[width = 0.23\textwidth]{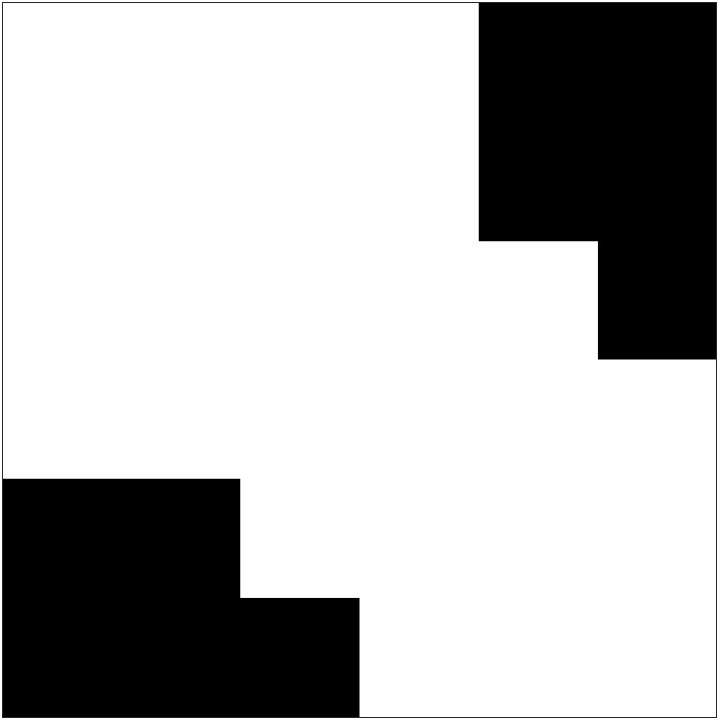} \\
    \end{tabular}
    \caption{From left to right: $G_6$, $A(G_6)$, the resultant $\mathbb{R}^2$ vectors, $Q$.}
    \label{fig:G6}
\end{figure}

Here, the last two eigenvectors of $2Q-J$ are $x = (0.37,0.37,0.30,0.16,-0.16,-0.30)$ and $y = (0, 0, 0.26, 0.43, 0.43, 0.26)$, with front $\{v_1, v_2, v_3\}$, middle $\{v_4\}$, back $\{v_5, v_6\}$. The eigenvalues are $\approx -4.4940$ and $\approx -3.2361$ respectively, giving $\frac{\lambda_{n-1} + \lambda_n}{n} \approx -0.6442$. 

\subsection{Proof of Theorem \ref{thm:n/2-regular}}

The spectrum of $J$ is $n/2, 0, \dots, 0$, all non-negative eigenvalues, hence by Corollary \ref{cor:invariant-spectrum}, it suffices to prove that $\inf\{(\mu_{n/2-1}(2Q-J) + \mu_{n/2}(2Q-J))/n\} > -\frac{\sqrt{2}}{2}$. Henceforth we work with $M = 2Q-J$ which adopts the same structure as $Q$ in Theorem \ref{thm:2Q-J}. For simplicity, we suppose $M$ now has $n$ vertices.

We adopt the same notation of vertex partitions and eigenvector Types. Slightly abusing notation, we still refer to $M_{ab} = +1$ as an edge and $M_{ab} = -1$ a non-edge.

We make most of the progress through the following theorem. Although we already knew the presented bound, we determine the equality case and later prove that it cannot occur even as a limit. 

\begin{thm}\label{thm:almost-n/2-regular}
    Suppose $G = G^*$ and $G$ is $n/2$-regular. Let $Q$ be constructed from its canonical rotation, then
    \[\frac{\mu_{n-1}(2Q-J) + \mu_n(2Q-J)}{n} \ge -\sqrt{2},\]
    with equality possible only if the front and back have size $n/2$ and the graph has $\frac{n^2}{8}$ edges.
\end{thm}

Due to the length of the full proof, we introduce our notation (which we continue to adopt in the subsequent proof of Theorem \ref{thm:n/2-regular}) and the underlying framework in multiple stages before finishing the proof in the final stage.

The stages proceed as follows. 
\begin{outline}[enumerate]
    \1 In Stage 1, we set up the algebra, with elements inspired by Csikv\'{a}ri \cite{csikvari2009conjecture}. In particular, we consider the algebraic formulation of the statement and extract equations relating the front, middle and back to the eigenvalues. 
    \1 In Stage 2, we use these equations alongside Corollary \ref{cor:invariant-degrees} and Theorem \ref{thm:2Q-J} to form inequalities that restrict the feasible region of the eigenvalues, which will depend on a set of parameters and two variables. 
    \1 In Stage 3, to analyse the minimum values of the eigenvalues, we consider the boundaries of the feasible regions, which correspond to cases of equality in the inequalities. 
    \1 In Stage 4, we bound the minimum of the feasible region when it occurs at an intersection of two boundaries. Furthermore, we find that in order to obtain $\frac{\mu_{n-1} + \mu_n}{n} = -\sqrt{2}$ at an intersection, the parameters and variables must take a specific set of values.
    \1 In Stage 5, we identify all the families of Type 1 feasible regions that can occur as we vary a given parameter, which motivates the casework in Stage 4. We also consider a Type 2 feasible region to aid the casework in the final stage. 
    \1 In the final stage, we use the above framework to deal with each case in order. \\
\end{outline}

\textbf{Stage 1: Setup.} Suppose we have front vertices $\{1, 2, \dots, k\}$, middle vertices $\{k+1, k+2, \dots, k+l\}$, and back vertices $\{k+l+1, \dots, n\}$. For $a \le k$, let $d_a$ denote the number of $c \ge k+l+1$ where $M_{ac} = +1$ and analogously define $d_c$, using the bipartite property from Theorem \ref{thm:invariant-bipartite}. Furthermore, with Corollary \ref{cor:invariant-degrees}, we have $d_a$ decreasing and $d_c$ increasing. Finally, note that eigenvectors are constant in their middle $l$ components. 

Consider a general eigenvector $x$ with eigenvalue $\mu < 0$. Let $A = \sum_{a = 1}^k x_a$, $B = \sum_{b = k+1}^{k+l} x_b = lx_b$ and $C = \sum_{c = k+l+1}^n x_c$. We have for $a \le k$,
\begin{align*}
    \mu x_a &= - \sum_{a = 1}^k x_a - \sum_{b = k+1}^{k+l} x_b - \sum_{\substack{c \ge k+l+1 \\ a \not \sim c}} x_c + \sum_{\substack{c \ge k+l+1 \\ a \sim c}} x_c \\
    &= -A - B- \sum_{\substack{c \ge k+l+1 \\ a \not \sim c}} x_c + \sum_{\substack{c \ge k+l+1 \\ a \sim c}} x_c \\
    \Rightarrow \mu \sum_{a = 1}^k x_a &= -kA - kB - \sum_{a = 1}^k \sum_{\substack{c \ge k+l+1 \\ a \not \sim c}} x_c + \sum_{a = 1}^k \sum_{\substack{c \ge k+l+1 \\ a \sim c}} x_c \\
    \Rightarrow \mu A + kA + kB &= \sum_{c = k+l+1}^n x_c(d_c - (k-d_c)) = -kC + 2\sum_{c = k+l+1}^n d_cx_c \\
    \Rightarrow \mu A + k(A+B+C) &= 2 \sum_{c = k+l+1}^n d_cx_c.
\end{align*}

We know the middle is non-empty, hence $l > 0$. By construction of the partition, we must have $\mu x_1 = - A - B + C$ and $\mu x_b = - A - B - C$. Hence, we get
\[\mu \left(A - k \frac{B}{l}\right) = 2 \sum_{c = k+l+1}^n d_cx_c, \quad \mu \left(C - (n-k-l) \frac{B}{l}\right) = 2\sum_{a = 1}^k d_ax_a,\]
where the latter follows by symmetry. Now let $t = \sum_{a = 1}^k d_a = \sum_{c = k+l+1}^n d_c$, and WLOG assume that $B \ge 0$. If $B < 0$, by Theorem \ref{thm:2Q-J}, $x$ must be Type 2 then reversing the order of the vertices and flipping the sign of $x$ gives $B > 0$. 

We first consider the $B = 0$ case, where $x$ still has to be Type 2. We have $B = 0$ which gives $x_b = 0$ and $A + C = 0$, then since $x_a \le x_1$ for $a \le k$, $\mu_2 C = 2\sum_{a = 1}^k d_a x_a \le 2x_1 \sum_{a=1}^k d_a = 4\frac{tC}{\mu_2}$. As $C \le 0$, we get $\mu_2 \ge \frac{4t}{\mu_2}$ which gives $\mu_2^2 \le 4t$ hence $\mu_2 \ge -2\sqrt{t}$. 

Henceforth assume $B > 0$. We also have $k, n-k-l \neq 0$ otherwise there are no edges and $M = -J$ with $\mu_{n-1} + \mu_n = -n > -\sqrt{2}n$. Let $A = \frac{kB}{l}a$ and $C = \frac{(n-k-l)B}{l}c$. Then our equations become:

\[x_1 = \frac{1}{\mu} \frac{B}{l} (-ka - l + (n-k-l)c), \quad -\mu = ka + l + (n-k-l)c,\]
\[\mu(a-1) = 2\frac{l}{kB} \sum_{c = k+l+1}^n d_cx_c, \quad \mu(c-1) = 2\frac{l}{(n-k-l)B} \sum_{a = 1}^k d_ax_a.\]
\vspace{3mm}

\textbf{Stage 2: Inequalities.} Given fixed $k, l, n, t$ we now consider the feasible region of the plane $-\mu = f(a, c) = ka + l + (n-k-l)c$. We impose four different types of inequalities. Let $\mu_1, \mu_2$ refer to Type 1, Type 2 eigenvalues respectively. 
\begin{outline}
    \1 \purp{Border}: For Type 1, by Theorem \ref{thm:2Q-J} we get $0 \le a, c \le 1$. For Type 2, we get $a \ge 1$ and $c \le 0$. 
    \1 \purp{Extrema}: For Type 2, we have $x_a \le x_1$ for $a \le k$, hence 
    \[\mu_2(c-1) \le 2\frac{l}{(n-k-l)B} x_1 \sum_{a = 1}^k d_a = 2\frac{t}{\mu_2(n-k-l)}(-ka - l + (n-k-l)c).\]
    \1 \purp{Mean}: We apply the Chebyshev sum inequality. For Type 1, note that $d_a$ is decreasing and $x_a$ is increasing. Thus, $(d_a), (x_a)$ are oppositely oriented, hence $\sum_{a=1}^k d_ax_a \le \frac1k \sum_{a=1}^k d_a \sum_{a=1}^k x_a = \frac{t}{k}A$. Thus, 
    \[\mu_1(c-1) \le 2\frac{t}{n-k-l}a.\]
    We also have $d_c$ increasing and $x_c$ decreasing, hence analogously
    \[\mu_1(a-1) \le 2\frac{t}{k}c.\]

    For Type 2, we have $d_a$ and $x_a$ decreasing, hence we obtain the opposite direction $\mu_2(c-1) \ge 2\frac{t}{n-k-l}a$. We still have $d_c$ increasing and $x_c$ decreasing, so again $\mu_2(a-1) \le 2\frac{t}{k}c$. 
    \1 \purp{Smoothing:} For Type 1, we perform smoothing given fixed $t, A$, where each $d_a \le n-k-l$. Indeed, given fixed $(x_a)_{a=1}^k$ as well, $\sum_{a=1}^k d_ax_a$ is lowest when $d_a = n-k-l$ for the smallest values of $a$. By considering the continuous generalisation, the minimum further occurs when $x_a = x_b$ for those with $d_a = 0$, since $A$ is fixed. Algebraically, this gives $d_a = n-k-l$ for $a \le \frac{t}{n-k-l}$ and $x_a = x_b$ for $a > \frac{t}{n-k-l}$. We get:
    \[\sum_{a=1}^k d_ax_a \ge (A - \left(k - \frac{t}{n-k-l}\right) x_b) (n-k-l) = \frac{B(n-k-l)}{l}\left(ka - k + \frac{t}{n-k-l}\right)\]
    \[\Rightarrow \mu_1(c-1) \ge 2(k(a-1) + \frac{t}{n-k-l}).\]
    Identically, we get $\mu_1(a-1) \ge 2((n-k-l)(c-1) + \frac{t}{k})$. 
\end{outline}

We normalise so that $X = \frac{k}{n}, Y = \frac{l}{n}, Z = \frac{n-k-l}{n}, T = \frac{t}{n^2}, \nu = \frac{\mu}{n}$.
\begin{outline}
    \1 Extrema: Since $\nu_2 < 0$, extrema(c) gives $\nu_2^2(c-1) \ge 2\frac{T}{Z}(-Xa - Y + Zc) = 2\frac{T}{Z}(\nu_2 + 2Zc)$. 
    \1 Mean: For Type 1, mean(c) gives $\nu_1(c-1) \le 2\frac{T}{Z}a$ and mean(a) gives $\nu_1(a-1) \le 2\frac{T}{X}c$. Type 2 gives $\nu_2(c-1) \ge 2\frac{T}{Z}a$ and $\nu_2(a-1) \le 2\frac{T}{X}c$. 
    \1 Smoothing: smoothing(c) gives $\nu_1(c-1) \ge 2X(a-1) + \frac{2T}{Z} = 2Xa + (\frac{2T}{Z} - 2X)$, and smoothing(a) gives $\nu_1(a-1) \ge 2Z(c-1) + \frac{2T}{X} = 2Zc + (\frac{2T}{X} - 2Z)$. \\
\end{outline}
As examples, consider $G_4$ and $G_6$ from above. We show these inequalities visually with Desmos, plotting $a$ on the x-axis, $c$ on the y-axis and $\nu$ on the z-axis. 

{\bf Legend:} Refer to the following colouring scheme for the boundaries: (red, mean(a)), (orange, mean(c)), (grey, extrema(c)), (blue, smoothing(a)), (purple, smoothing(c)). 

The green then represents the feasible region for $\mu_1$ of Type 1 and $\mu_2$ of Type 2 from the inequalities, and the green balls denote the true values. For $G_4$, many boundaries coincide and the feasible region points to exactly the eigenvalues. 
\begin{figure}[H]
    \centering
    \begin{tabular}{cc}
         \includegraphics[width = 0.5\textwidth]{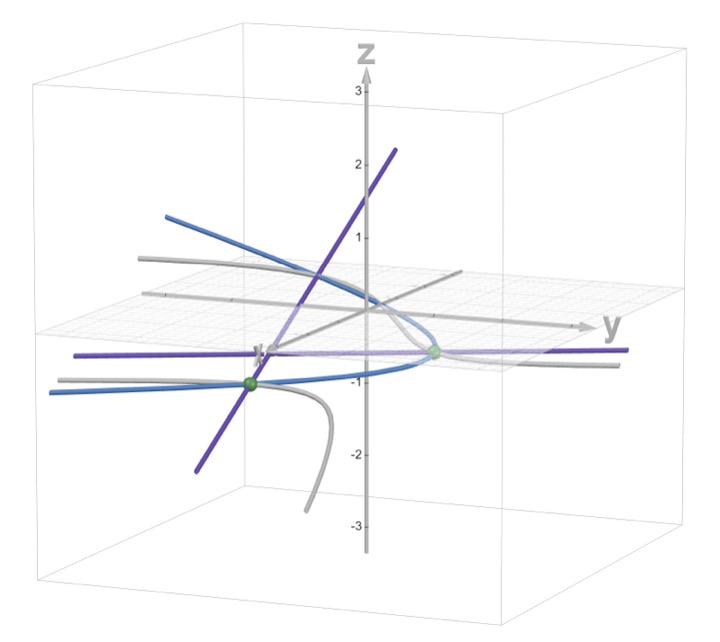} & \includegraphics[width = 0.45\textwidth]{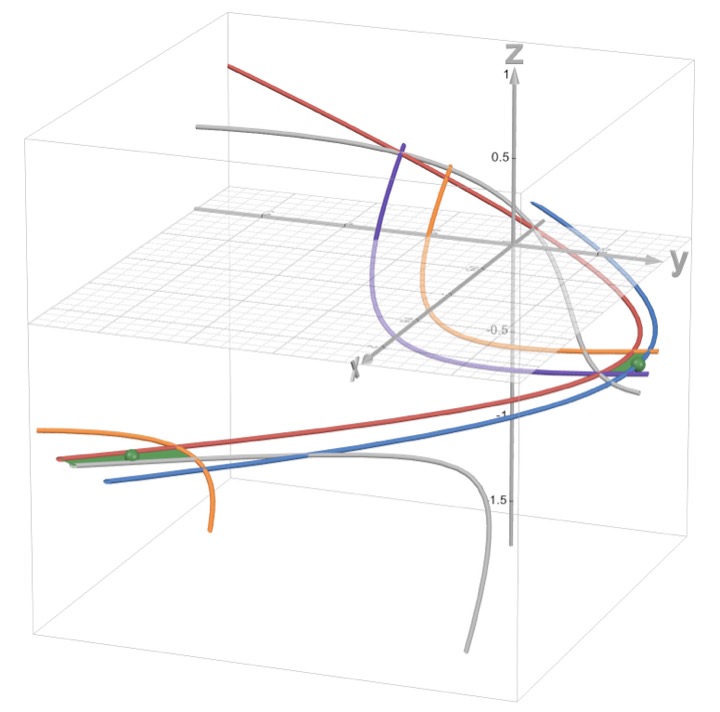} \\
    \end{tabular}
    \caption{Visually demonstrating the inequalities for $G_4$ (left) and $G_6$ (right).}
    \label{fig:plot-demonstration}
\end{figure}

\textbf{Stage 3: Boundaries.} Consider an intersection of an equation of the form $\nu(c-1) = pa + q$ with the plane. Noting that $\nu = -(Xa + Y + Zc)$, we get:
\begin{align*}
    pa + q &= -(c-1)(Xa + Y + Zc) \\
    \Rightarrow 0 &= Zc^2 + Xac + (p-X)a + (Y-Z)c + (q-Y) \\
    \Rightarrow a &= - \frac{Zc^2 + (Y-Z)c + (q-Y)}{Xc + p - X} \\
    \Rightarrow \nu &= -\frac{Zpc + Yp - Xq}{Xc + p - X} = -\frac{Zp}{X} + \frac{Xq + \frac{Zp^2}{X} - Zp - Yp}{Xc + p - X}
\end{align*}
which defines a hyperbola $\nu = f(c)$. Consider the non-constant part:
\begin{outline}
    \1 mean(c): $p = 2\frac{T}{Z}$ and $q = 0$, hence $2\frac{T}{Z}\frac{\frac{2T}{X} + X - 1}{Xc + 2\frac{T}{Z} - X}$. Similarly, $2\frac{T}{X} \frac{\frac{2T}{Z} + Z - 1}{Za + 2\frac{T}{X} - Z}$ for mean(a). 
    \1 smoothing(c): $p = 2X$, $q = 2\frac{T}{Z}-2X$, hence $2\frac{\frac{T}{Z}+2Z-1}{c+1}$. $2\frac{\frac{T}{X}+2X-1}{c+1}$ for smoothing(a). 
\end{outline}
Fix $X, Z$ and consider the evolution of the inequalities as $T$ varies from $0 \to XZ$. Let an inequality \purp{flip} at a value $T$ if it flips the sign of the numerator of the non-constant part of the hyperbola going from $T^- \to T^+$. 

We have that mean(c) flips from \purp{negative} to \purp{positive} (describing the numerators) at $T = \frac{X(1-X)}{2}$, and similarly mean(a) flips from negative to positive at $T = \frac{Z(1-Z)}{2}$. Then smoothing(c) flips from negative to positive at $T = Z(1-2Z)$ and smoothing(a) at $T = X(1-2X)$. Note that if $X \ge 1/2$ or $Z \ge 1/2$ then we could have smoothing always positive. On the other hand, mean always starts off negative. \\

In later stages we often treat $\nu$ as a function of $a$ or $c$ only, and label the hyperbola boundaries as increasing or decreasing. Note that the feasible region defined just by one inequality of the given form constrained on the plane will always either lie above the upper branch and below the lower branch (i.e. inside the individual branches), or below the upper branch and above the lower branch (i.e. between the branches) when projected onto the respective $(c,\nu)$ or $(a,\nu)$ plane. 

This is determined by replacing the equals signs above by the correct inequality sign, which depends on the value of $c$ (or $a$) since we divide through by $Xc + 2\frac{T}{Z} - X$ (or the $a$ equivalent). We demonstrate with Type 1 smoothing(c). Since $c + 1 > 0$ always, we have $pa + q \le -(c-1)(Xa + Y + Zc)$ hence $0 \ge Zc^2 + Xac + (p-X)a + (Y-Z)c + (q-Y)$ and
\[a \le -\frac{Zc^2 + (Y-Z)c + (2\frac{T}{X} - 2X - Y)}{Xc + X},\]
which implies
\[\nu_1 \ge -2Z + 2\frac{\frac{T}{Z} + 2Z - 1}{c+1}.\]

For $T < Z(1-2Z)$, smoothing(c) is increasing and negative with the feasible region lying above the lower branch. Without restriction on $c$, it would lie between the branches. Analogously, for $T > Z(1-2Z)$, smoothing(c) is decreasing and positive with the feasible region lying above the upper branch. Without restriction on $c$, it would then lie inside the individual branches. For $T = Z(1-2Z)$, smoothing(c) is constant. 

The others follow similarly. Note that for mean inequalities, we could have both branches satisfying the border(a) and border(c) inequalities. \\

\textbf{Stage 4: Intersections.} We find that the minimum $\nu$ often occurs at the intersection of two such inequalities. We deal mainly with Type 1 and explicitly say when we touch on Type 2. The cases that we consider arise from the analysis of the phases in Stage 5. \\

\purp{Intersection I}: At the intersection of mean(c) and smoothing(c), 
\[2\frac{T}{Z}a = \nu_1(c-1) = 2\left(X(a-1) + \frac{T}{Z}\right),\]
which gives $(\frac{T}{Z} - X)(a-1) = 0$. Note that $\frac{T}{Z} = X$ iff we have a complete bipartite graph between the front and back, which yields a block construction and is dealt with as the complement of the construction in Theorem \ref{thm:block}. Since this is regular, we get $\lambda_{n-1} + \lambda_n \ge -\frac{2n}{3}$ which implies $\nu_1 + \nu_2 \ge -\frac{4}{3} > -\sqrt{2}$. 

Otherwise, we have $a = 1$. Then $\nu_1(c-1) = 2\frac{T}{Z}$, which yields $c = 2\frac{T}{Z\nu_1} + 1$ and $-\nu_1 = Xa + Y + Zc = 2\frac{T}{\nu_1} + 1$ where $X + Y + Z = 1$. Hence we get $\nu_1^2 + \nu_1 + 2T = 0$, with intersections at $\frac{-1\pm \sqrt{1-8T}}{2}$.     

Note that these intersections may not always exist. Indeed, we require $1-8T \ge 0$ or equivalently $T \le \frac18$. But if this holds and $\nu_1 < 0$, then we get $a = 1, c = 2 \frac{T}{Z\nu_1} + 1$ yielding a valid intersection point. Thus, mean(c) and smoothing(c) intersect until $T = \frac18$, similarly with mean(c) and smoothing(c).\\

\purp{Intersection Ia}: First we consider when the minimum $\nu_1$ occurs at the upper intersection. Hence $\nu_1 \ge \frac{-1+\sqrt{1-8T}}{2}$ which is always at least $-\frac12$. But
\[\sum_{i = 1}^n \mu_i^2 = Tr(M^2) = \sum_{i = 1}^n \sum_{j = 1}^n M_{ij} M_{ji} = n^2,\]
thus $\nu_1^2 + \nu_2^2 \le 1$ and $\nu_1 + \nu_2 \ge -\frac12 - \frac{\sqrt{3}}2 > -\sqrt{2}$.\\

\purp{Intersection Ib}: Otherwise, if we have a Type 2 $\nu_2 \ge -2\sqrt{T}$ and a Type 1 $\nu_1 \ge -\frac{1+\sqrt{1-8T}}{2}$, then $\nu_1 + \nu_2$ is bounded below by $-2\sqrt{T} - \frac{1+\sqrt{1-8T}}{2}$ at $T = \frac1{12}$ which gives $-\frac{1+\sqrt{3}}{2} > -\sqrt{2}$. This also holds for the $a$ variants. \\

\purp{Intersection II}: We next consider the intersection of smoothing(a) and smoothing(c) both positive. We get:
\begin{align*}
    a-1 &= \frac{2Z}{\nu_1} (c-1) + \frac{2T}{X\nu_1} \\
    \Rightarrow \nu_1(c-1) &= \frac{4XZ}{\nu_1} (c-1) + \frac{4T}{\nu_1} + \frac{2T}{Z} \\
    \Rightarrow \left(\nu_1 - \frac{4XZ}{\nu_1}\right)(c-1) &= \frac{4T}{\nu_1} + \frac{2T}{Z} \\
    \left(\nu_1 - \frac{4XZ}{\nu_1}\right)(a-1) &= \frac{4T}{\nu_1} + \frac{2T}{X} \\
    \Rightarrow -\left(\nu_1 - \frac{4XZ}{\nu_1}\right)\nu_1 &= \frac{4TX}{\nu_1} + 2T + \frac{4TZ}{\nu_1} + 2T + \left(\nu_1 - \frac{4XZ}{\nu_1}\right) \\
    \Rightarrow 0 &= P(\nu_1) = \nu_1^3 + \nu_1^2 + 4(T-XZ)\nu_1 + 4(TX + TZ - XZ).
\end{align*}

Note that $P(0) = 4(TX + TZ - XZ) \le 0$ since $T \le XZ$ and $X + Z \le 1$, hence there is at least one non-negative root. For reasons justified in Stage 5, we require the second root of this cubic, which we show exists. 

This will always yield a valid intersection point, since from $\nu_1$ we obtain $a = \frac{\frac{4T}{\nu_1} + \frac{2T}{X}}{\nu_1 - \frac{4XZ}{\nu_1}} + 1$ and similarly for $c$. This is well-defined unless $\nu_1 - \frac{4XZ}{\nu_1} = 0$ or equivalently $\nu_1 = \pm 2\sqrt{XZ}$. But then we require $0 = P(\nu_1) = 4T(X + Z \pm 2\sqrt{XZ})$ hence we have the equality case of AM-GM, where $X = Z$ and $\nu_1 = -2X$. In that case, $-2X(a-1) = 2X(c-1) + \frac{2T}{X}$ which implies $a + c = -\frac{T}{X^2} + 2$, and furthermore $2X = -\frac{T}{X} + 2X + Y$ gives $XY = X(1-2X) = T$. This intersection is a line, so we can pick any point, e.g. $a = c = -\frac{T}{2X^2} + 1$ which suffices since $T \le X^2$. 

Define $S = XZ$. Fix $T, S$ and vary $Z$. Note that only the constant term changes, and to get the minimal second root, we want to minimise this constant term, equivalent to minimising $Z + \frac{S}{Z}$. We already have $Z + \frac{S}{Z} \le 1$ and $\frac14 \ge S \ge T$. \\

\purp{Intersection IIa}: Our first case involves the additional restriction $Z \le \sqrt{T}$, alongside the assumed $T \ge Z(1-2Z)$ which implies $Z \le \frac{1 - \sqrt{1-8T}}{2}$ or $Z \ge \frac{1+\sqrt{1-8T}}{2}$. Yet $\frac{1+\sqrt{1-8T}}{2} > \sqrt{T}$ hence we must have the former.

This also means $Z \le \sqrt{S}$ hence $Z \le \frac{S}{Z}$. Since $Z + \frac{S}{Z} \le 1$, we have $Z^2 - Z + S \le 0$ which gives $Z \ge \frac{1-\sqrt{1-4S}}{2}$ and subsequently $\sqrt{T} \ge \frac{1-\sqrt{1-4S}}{2}$. We get that $Z + \frac{S}{Z}$ is minimised if we move $Z$ closer to $\sqrt{S}$, hence minimum occurs at $Z = \min\{\sqrt{T}, \frac{1-\sqrt{1-8T}}{2}\}$. In particular, the former for $T \ge \frac19$, and the latter for $T \le \frac19$. 

Our polynomial becomes, in the two cases respectively,
\[P(\nu_1) = \nu_1^3 + \nu_1^2 + 4(T-S)\nu_1 + 4(\sqrt{T}(T + S) - S),\]
\[P(\nu_1) = \nu_1^3 + \nu_1^2 + 4(T-S)\nu_1 + 4\left(\frac{2TS}{1-\sqrt{1-8T}} + \frac{T(1-\sqrt{1-8T})}{2} - S\right).\]

Suppose that the eigenvalue satisfies $\nu_1 \le -0.7$. Then $P(-0.7) \le 0$ and $-0.7$ is past the second turning point of $P$, which occurs when $3\nu_1^2 + 2\nu_1 + 4(T-S) = 0$ where $\nu_1 = \frac{-1-\sqrt{1-12(T-S)}}{3}$. Thus, $-1.1 \ge -\sqrt{1-12(T-S)}$, which rearranges to $1.21 \le 1-12(T-S)$ and subsequently we obtain $S \ge T + \frac{7}{400}$. 

Consider $P(-0.7)$ with fixed $T$ and varying $S$. In both cases, $S$ is linear hence attains its minimum at an endpoint, either at the minimum $S = T + \frac{7}{400}$ or at a maximum when $\sqrt{T} = \frac{1-\sqrt{1-4S}}{2}$, giving $S = \frac{1 - (2\sqrt{T}-1)^2}{4}$ (if this upper bound is lower than the lower bound, contradiction). We substitute both in to the first polynomial if $T \ge \frac19$ and the second if $T \le \frac19$, then apply the boundary condition $0 \le S \le \frac14$. Directly computing with Desmos, we get that as functions of $T$, we have $P(-0.7)$ is always positive, contradiction. Thus, $\nu_1 > -0.7$. Again, $\nu_1^2 + \nu_2^2 \le 1$ means $\nu_1 + \nu_2 \ge -0.7 - \sqrt{1-0.7^2} > -\sqrt{2}$. \\

\purp{Intersection IIb}: Our second case doesn't restrict $Z \le \sqrt{T}$ but instead claims that for Type 2, we have $\nu_2 \ge -2\sqrt{T}$. The minimum for $Z + \frac{S}{Z}$ occurs when $Z = \sqrt{S}$, whence the polynomial becomes
\[P(\nu_1) = \nu_1^3 + \nu_1^2 + 4(T-S)\nu_1 + 4(2T\sqrt{S} - S).\]

Suppose that $\nu_1 + \nu_2 \le -\sqrt{2}$, where $\nu_1$ is the second root for some $T, S$. Then $\nu_1 - 2\sqrt{T} \le -\sqrt{2}$. We again require $P(-\sqrt{2} + 2\sqrt{T}) \le 0$ and that $-\sqrt{2} + 2\sqrt{T}$ is past the second turning point of $P$. Thus, as before we get $-\sqrt{2} + 2\sqrt{T} \ge \frac{-1-\sqrt{1-12(T-S)}}{3}$ which gives $-3\sqrt{2} + 6\sqrt{T} + 1 \ge -\sqrt{1-12(T-S)}$. Since $T \le \frac14$, $-3\sqrt{2} + 6\sqrt{T} + 1 \le 4-3\sqrt{2} < 0$, hence
\[S \ge T + \frac{(-3\sqrt{2} + 6\sqrt{T} + 1)^2 - 1}{12}.\]

Consider $P(-\sqrt{2} + 2\sqrt{T})$ with fixed $T$ and varying $S$. We have derivative $4\sqrt{2} - 4 + \frac{4T}{\sqrt{S}}$, which is decreasing in $S$. Thus, it attains its minimum at an endpoint, either the above minimum or $S = \frac14$. We can compute that the proposed lower bound is greater than $\frac14$ for $T < 0.055$, hence assume $T \ge 0.055$. Substituting the lower bound, we get that as a function of $T$, $P(-\sqrt{2} + 2\sqrt{T})$ is bounded below by a positive number (0.001 suffices). Substituting the upper bound, we always get non-negative values with the only root at $T = \frac18$. Thus, we must have exactly $X = \frac12, Z = \frac12, T = \frac18$ with $\nu_1 = \nu_2 = -\frac{\sqrt{2}}{2}$. \\

\textbf{Stage 5: Phases.} We consider possible phases of the Type 1 feasible region as $T$ increases from $0 \to XZ$. For the graphical visualisations of the feasible regions, refer to the same colouring scheme as above, copied here for ease: (red, mean(a)), (orange, mean(c)), (grey, extrema(c)), (blue, smoothing(a)), (purple, smoothing(c)), and green for the resultant feasible region. All displayed states correspond to $X = 0.75, Z = 0.2$. 
\begin{outline}[enumerate]
    \1 \purp{Phase 1c}: mean(c), smoothing(c) both negative.

    \begin{figure}[H]
        \centering
        \begin{tabular}{cc}
            \includegraphics[width = 0.5\textwidth]{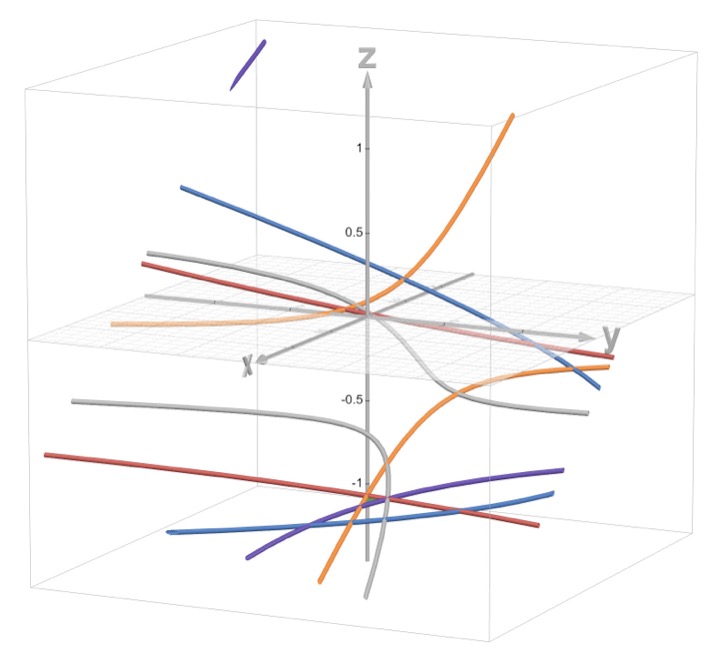} & \includegraphics[width = 0.5\textwidth]{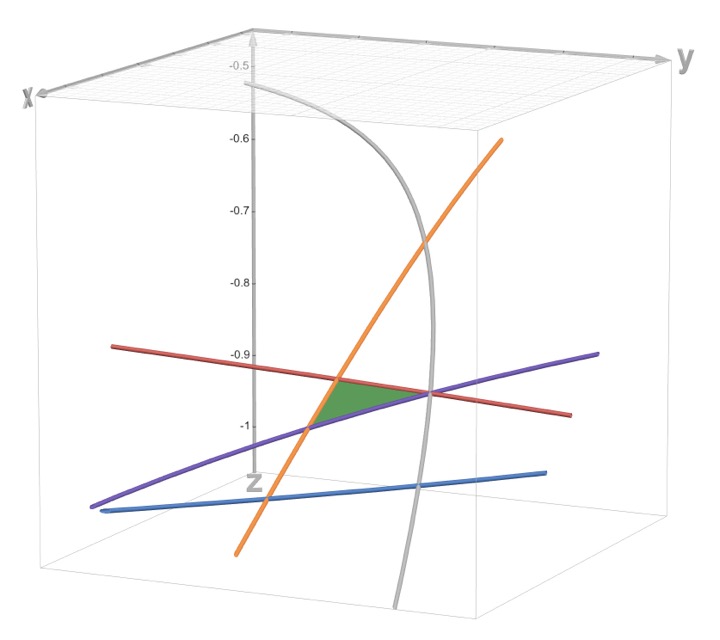}\\
        \end{tabular}
        \caption{Phase 1 visualised, with $T = 0.05$. Overall and close-up view.}
        \label{fig:phase1}
    \end{figure}

    We have that smoothing(c) and mean(c) are increasing hyperbolas in $c$. Note that $c+1>0$ hence the region $0 \le c \le 1$ from border(c) doesn't intersect with the upper branch of smoothing(c), and hence the feasible region lies above the lower branch. It also lies either above the upper branch of mean(c) or below the lower branch, therefore the minimum must occur at the lower intersection of the two lower branches. This corresponds to $\nu_1 = -\frac{1+\sqrt{1-8T}}{2}$ from Intersection I in Stage 4.  
    \1 \purp{Phase 2a}: mean(a), smoothing(a) both positive, with smoothing(c) negative.

    \begin{figure}[H]
        \centering
        \begin{tabular}{cc}
            \includegraphics[width = 0.5\textwidth]{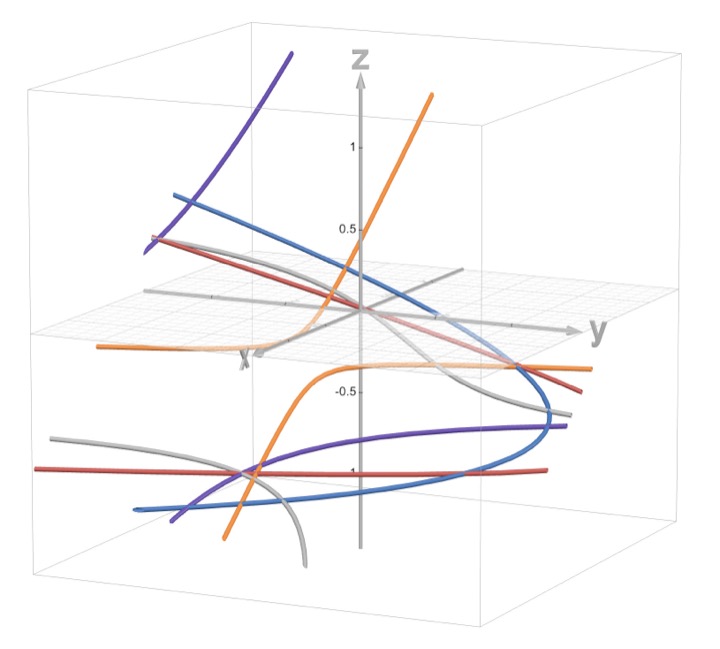} & \includegraphics[width = 0.5\textwidth]{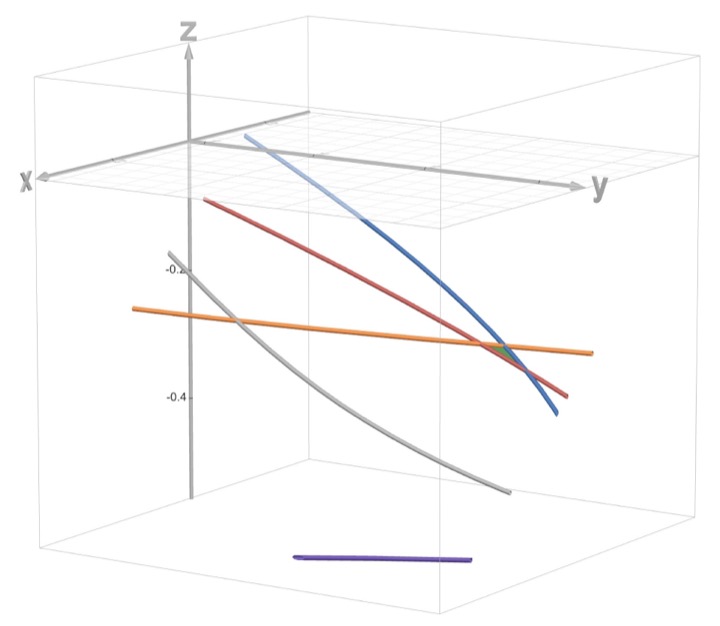}\\
        \end{tabular}
        \caption{Phase 2 visualised, with $T = 0.10$. Overall and close-up view.}
        \label{fig:phase2}
    \end{figure}
    
    $T \ge \frac{Z(1-Z)}{2}$ hence $2\frac{T}{Z} - X \ge 1 - Z - X \ge 0$. Now only the upper branch of mean(a) intersects with $a \ge 0$ since $a \ge 0 \ge 1-\frac{2T}{XZ}$ implies $Za + \frac{2T}{X} - Z \ge 0$, and the feasible region must lie outside it. Similarly, $a+1>0$ hence the feasible region lies inside the upper branch of smoothing(a). 

    We claim that smoothing(c) eliminates anything below the lower intersection of these two branches, hence the minimum occurs at the top intersection since they're both decreasing in $a$, where $\nu_1 = \frac{-1+\sqrt{1-8T}}{2}$. 

    Indeed, since smoothing(c) has $\nu_1$ increasing with $c$ for $T > Z(1-Z)$, at $c = 1$ it beats the value at its intersection with mean(c), where $a = 1$, $\nu_1 = -\frac{1+\sqrt{1-8T}}{2}$ and $c = 2\frac{T}{2\nu_1} + 1 = 1-\frac{4T}{Z(1+\sqrt{1-8T})} < 1$. In particular, it is higher than the second intersection of mean(a), smoothing(a) at $c = 1$. Yet we also have at $c = 0$, $Xa + Y = -\nu = 2Xa + \frac{2T}{Z} - 2X$ which gives $a = \frac{Y+2X-\frac{2T}{Z}}{X} \le 1$ since $Y + X - \frac{2T}{Z} \le Y + X - (1-Z) = 0$. Thus, smoothing(c) is also higher than mean(a) at $c = 0$, since $\nu = -Xa - Y$ is decreasing in $a$. 

    We have that the smoothing(c) boundary is concave and increasing in $c$, since it is the lower branch of an increasing hyperbola. The feasible region lies above this. Furthermore, below the lower intersection, the mean(a) boundary is also increasing but convex in $c$ since there exists a symmetric, higher intersection on the same branch, and here mean(a) bounds the lower feasible region from above. Thus, the complement of smoothing(c) encloses everything below the lower intersection. 
    \1 \purp{Phase 3}: smoothing(a), smoothing(c) both positive. 

    \begin{figure}[H]
        \centering
        \begin{tabular}{cc}
            \includegraphics[width = 0.5\textwidth]{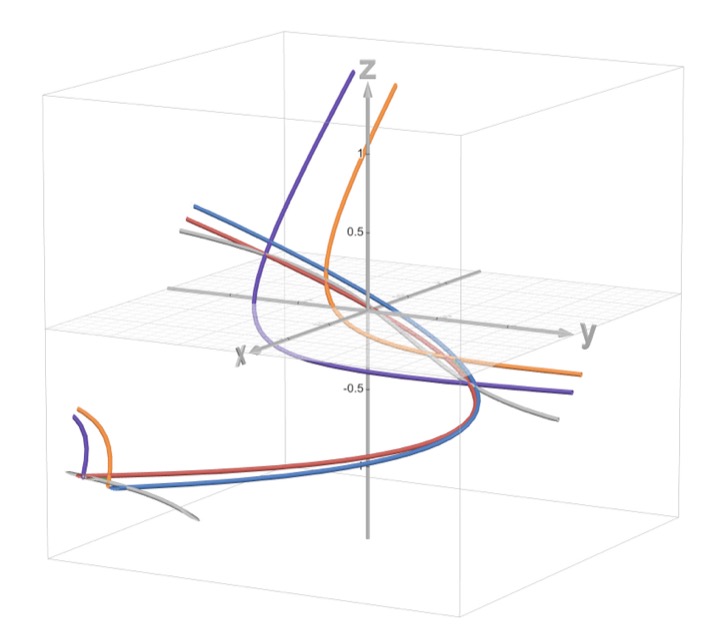} & \includegraphics[width = 0.5\textwidth]{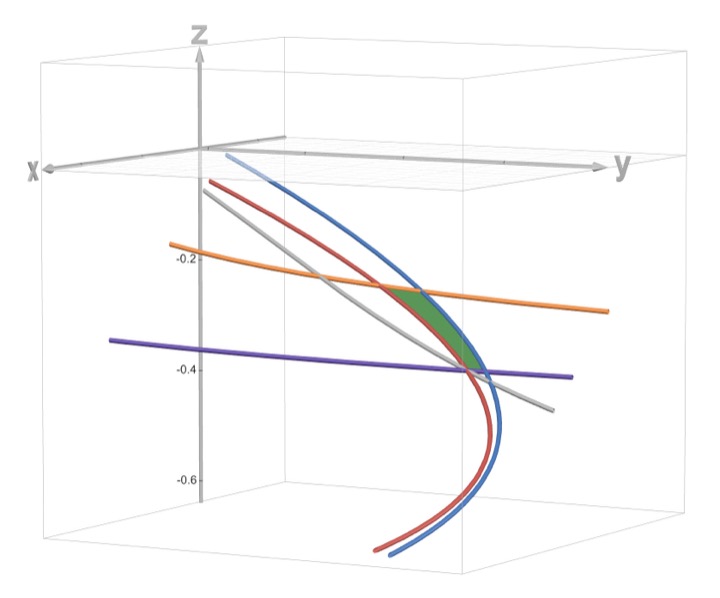}\\
        \end{tabular}
        \caption{Phase 3 visualised, with $T = 0.12$. Overall and close-up view.}
        \label{fig:phase3}
    \end{figure}

    Since $a+1 > 0$ and $c+1 > 0$, we only have the upper branches of both hyperbolas intersecting with the border conditions $a, c \ge 0$, and the feasible region lies above (or inside) both. Thus, it must lie inside the region bounded by the two intersection points of the branches, necessarily the two highest ones. Since the region bounded within a hyperbola branch is convex, the intersection of these two inequalities gives a convex region, minimised at an extreme point. In this case, the lower intersection point yields the minimum. \\
\end{outline}

For Type 2, we separately consider one possible phase, which exists concurrently with one of the Type 1 phases above. 

\begin{outline}[enumerate]
    \1 \purp{Phase $\Phi$}: mean(a) negative. 

    Now $a \ge 1$ hence $a \ge 1-\frac{2T}{XZ}$, so the feasible region lies below the lower branch. Yet mean(a) is increasing in $a$ hence is bounded below by the value of $\nu_2$ as $a \to \infty$. For finite $\nu_2$ we get $-\nu_2 = Xa + Y + Zc$ which yields $c \sim -\frac{X}Z a$, hence $\nu_2 = \frac{2T}{X} \frac{c}{a-1} \sim \frac{2T}{X} \frac{c}{a} = -\frac{2T}{Z}$. Thus, $\nu_2 < -\frac{2T}{Z}$. \\
\end{outline}

\begin{proof}[Proof of Theorem \ref{thm:almost-n/2-regular}]
    We now run through the various cases. At the end of one phase and the start of another, when $T$ exactly equals a threshold, we note that the arguments for the Phase 1 and Phase 3 bounds still hold. This is because a degenerate hyperbola is both increasing and decreasing, hence we can replace `positive' with `non-negative' and `negative' with `non-positive' in these instances above. Furthermore, if $T = 0$ we again get no edges and $\nu_1 + \nu_2 \ge -1$, and if $T = XZ$ we get the block construction with $\nu_1 + \nu_2 \ge -\frac43$. 

    For Type 2, we have two distinct cases. Note that extrema(c) gives $c(\nu_2^2 - 4T) \ge \nu_2^2 + \frac{2T}{Z}\nu_2$. Since $c \le 0$, if we have $\nu_2 < -2\sqrt{T}$, then $\nu_2^2 - 4T \ge 0$ hence $0 \ge \nu_2(\nu_2 + \frac{2T}{Z})$ and therefore $\nu_2 \ge -\frac{2T}{Z}$. Thus, either $\nu_2 \ge -2\sqrt{T}$ or $\nu_2 \ge -\frac{2T}{Z}$, where the latter requires $-2\sqrt{T} > -\frac{2T}{Z}$ which rearranges to $T > Z^2$. For $B = 0$, we proved above that we always have $\nu_2 \ge -2\sqrt{T}$. 
    \begin{outline}
        \1[-] If $1-2Z \ge X$, then $Z(1-2Z) \ge XZ \ge T$ and $\frac{1-X}{2} \ge Z$ which yields $\frac{X(1-X)}{2} \ge XZ \ge T$. Thus, mean(c) and smoothing(c) are always negative, and we always have Phase 1c. 
        \2[-] If $T \le Z^2$, then $\nu_2 \ge -2\sqrt{T}$. This is Intersection Ib.
        \2[-] If $T > Z^2$ and $\nu_2 \ge -\frac{2T}{Z}$, then consider mean(a). If $T < \frac{Z(1-Z)}{2}$, then we have Phase $\Phi$ and $-\frac{2T}{Z} > \nu_2 \ge -\frac{2T}{Z}$, contradiction. Thus we must have $T \ge \frac{Z(1-Z)}{2}$. 
        
        If $T \ge X(1-2X)$, then we have Phase 2a and Intersection IIa, noting that since $T \le \frac{X(1-X)}{2} \le \frac18$, these intersections exist. Otherwise, we get $\frac{Z(1-Z)}{2} \le T < \frac{2X(1-2X)}{2}$ hence $2X \in [Z, 1-Z]$. But then $2X \le 1-Z$ gives $2X + Z \le 1$ and we reduce to the case below.
        \1[-] If $1-2X \ge Z$, then identically mean(a), smoothing(a) are always negative and we always have Phase 1a. We also always have Phase $\Phi$ and in particular we cannot have $\nu_2 < -2\sqrt{T}$, since again $-\frac{2T}{Z} > \nu_2 \ge -\frac{2T}Z$ yields a contradiction. Thus, we have Intersection Ib. 
        \1[-] Henceforth assume $X+2Z > 1$ and $2X+Z > 1$. 
        \1[-] If $X, Z \le \frac12$, either $T < Z^2$ and we have a universal $\nu_2 \ge -2\sqrt{T}$ bound, or $\nu_2 \ge -\frac{2T}{Z}$ only if $XZ \ge T \ge Z^2$, which gives $X \ge Z$. Thus, WLOG suppose $X \ge Z$. 

        We cannot have $Z \le \frac14$ otherwise $X + 2Z \le 1$, contradiction. In particular, $2Z > X > 1-2Z$ and $\frac{X(1-X)}{2} > \frac{2Z(1-2Z)}{2} = Z(1-2Z)$. Thus, mean(c) flips after smoothing(c). But because $\frac14 < Z \le X \le \frac12$, we further have $Z(1-2Z) \ge X(1-2X)$ hence smoothing(a) flips before smoothing(c), giving the following progression of flips in order:
    
        \centerline{smoothing(a) $\to$ smoothing(c) $\to$ mean(c).}
        \2[-] If $T \le Z^2$, then we have Phase 1c, Intersection Ib until smoothing(c) flips, after which we get Phase 3, Intersection IIb.
        \2[-] If $T > Z^2$ and $\nu_2 \ge -\frac{2T}{Z}$, note that $\frac{Z(1-Z)}{2} \ge X(1-2X)$ since $2X + Z > 1$, and furthermore $\frac{X(1-X)}{2} \ge \frac{Z(1-Z)}{2}$. Thus, mean(a) flips between smoothing(a) and mean(c). 

        Consider after mean(a) flips, since before gives Phase $\Phi$ and the same contradiction. Until smoothing(c) flips, we have Phase 2a hence Intersection Ia, which again exists since mean(c) hasn't flipped so $T \le \frac18$. After smoothing(c) flips, we have Phase 3 hence Intersection IIa. 
        \1[-] Now again WLOG we can suppose $X > \frac12$ and $Z < \frac12$. In particular, $1-X, X \ge Z$ hence $\frac{X(1-X)}{2} \ge \frac{Z(1-Z)}{2}$ and mean(c) flips after mean(a). Note we also have smoothing(a) always positive. 
        \2[-] If $X \le 2Z$, again we have mean(c) flips after smoothing(c). We have the same phase progression as just above. 
        \2[-] Else, $X > 2Z$ and mean(c) flips before smoothing(c). The flip progression is:

        \centerline{smoothing(a) $\to$ mean(a) $\to$ mean(c) $\to$ smoothing(c).}

        We have Phase 1c until mean(a) flips, then Phase 2a until smoothing(c) flips, afterwards Phase 3. Note that the intersections for Phase 2a exist until $T = \frac18$ which is past smoothing(c) flipping. 
        \3[-] If $T \le Z^2$, then we have Intersection Ib, Intersection Ia, then Intersection IIb. 
        \3[-] If $T > Z^2$ and $\nu_2 \ge -\frac{2T}{Z}$, we skip until mean(a) flips, then we have Intersection Ia, IIa.
    \end{outline}
    
    Finally, note that $\nu_1 + \nu_2 = -\sqrt{2}$ means we must have Intersection IIb, $X = Z = \frac12$, $T = \frac18$ and $\nu_1 = \nu_2 = -\frac{\sqrt{2}}{2}$. For Type 1 this requires $a = c = \frac{\sqrt{2}}{2}$, and for Type 2 we have either $a \to \infty, c \to -\infty$ or equivalently, $B = 0$.
\end{proof}

Suppose FTSOC there exists a sequence of matrices $M_n$ such that $\nu_1(M_n) + \nu_2(M_n) \to -\sqrt{2}$ as $n \to \infty$. By continuity of all inequalities and the fact that equality only occurs at this specific state, with $\nu_1 + \nu_2$ from Intersections Ia, Ib and IIa all bounded a positive constant away from $-\sqrt{2}$, the matrices get arbitrarily close to this \purp{equality state}. In particular, as $n \to \infty$, we get $X, Z \to \frac12$, $T \to \frac18$, $\nu_1, \nu_2 \to -\frac{\sqrt{2}}{2}$ and $a, c \to \frac{\sqrt{2}}{2}$. \\

Further suppose FTSOC that there exists $N \in \mathbb{N}$ such that there exists arbitrarily large $n$ for which $|V(M_n)| \le N$. There are finitely many possible $M_n$, hence by pigeonhole, there exists $M$ with $|V(M)| = m \le N$ occurring infinitely many times. In particular we must have equality in all the quantities above. But then $\mu_1^2 + \mu_2^2 = m^2$ implies $\mu_3 = \dots = \mu_m = 0$ since $\sum_{i=1}^m \mu_i^2 = m^2$. Yet $-\sqrt{2}m = Tr(M) = -m$ contradiction. In particular, we get $|V(M_n)|$ is unbounded and can WLOG assume $|V(M_n)| = n$ for infinitely many $n$. We focus on this subsequence to simplify the notation in the following claims, but the arguments themselves hold as long as $|V(M_n)|$ is unbounded. \\

We consider limits of certain algebraic averages common to each graph as $n \to \infty$, which doesn't require that the sequence $(M_n)$ itself converges. Our goal is to formalise the fact that equality for both Type 1 and Type 2 cannot simultaneously occur.

Let the \purp{top} of the front refer to the first $\frac{t}{n-k-l}$ vertices in the front, and the \purp{bottom} of the front the remaining $k - \frac{t}{n-k-l}$. Given an interval $I$, let $nI$ denote $\{nr \mid r \in I\}$. 

\begin{clm}\label{clm:degree}
    For all $\varepsilon, \delta > 0$, there exists $N \in \mathbb{N}$ such that $\forall n \ge N$, in $M_n$ we have $d_a \le (\frac12 - \frac{\sqrt{2}}{4} + \varepsilon)n$ for all $a \in n(\frac14 + \delta, \frac12]$. 
\end{clm}
\begin{proof}
    We reconsider the derivation of smoothing(c) for the Type 1 eigenvector. WLOG $x_b = \frac{B}{l} = 1$. Note that $a$ converging to $\frac{\sqrt{2}}{2}$ represents the average of the front $x_a$. Suppose FTSOC that there exists $\epsilon > 0$ such that there are arbitrarily large $n$ with the average of the bottom of the front $x_a$ at most $1-\epsilon$. Then $\sum x_ad_a \ge (A - (k - \frac{t}{n-k-l})x_b(1-\epsilon))(n-k-l)$ which turns smoothing(c) into $\nu_1(c-1) \ge 2Xa + (\frac{2T}{Z} - 2X)(1-\epsilon)$. Yet this gives $\nu_1 + \nu_2$ strictly greater than $-\sqrt{2}$ by some $\delta(\epsilon) > 0$, since the equality state is infeasible and the infimum is attained by continuity. Thus, for sufficiently large $n$, the average of the bottom of the front $> 1-\epsilon$ and hence converges to 1. 

    Furthermore, since $x_a \le 1$ we must further have that any constant positive proportion of the bottom has average converging to 1. Additionally, the average of the top of the front converges to $\sqrt{2} - 1$, since $\frac{k - \frac{t}{n-k-l}}{n} = X - \frac{T}{Z} \to \frac14$. The reverse holds for the back due to smoothing(a). 

    Let $I$ be a non-trivial sub-interval of $(1/4, 1/2)$. Then because $\frac{k - \frac{t}{n-k-l}}{n} \to \frac14$ and $\frac{k}{n} \to \frac12$, we have that for sufficiently large $n$, $nI \subseteq [k - \frac{t}{n-k-l}, k]$. Thus, with $r = |nI \cap \mathbb{N}|$,
    
    \begin{align*}
        \mu_1 \sum_{a \in nI} x_a &= -r(A+B) + \sum_{a \in nI} \left(\sum_{\substack{c \ge k+l+1 \\ a \sim c}} x_c - \sum_{\substack{c \ge k+l+1 \\ a \not \sim c}} x_c\right) \\
        \Rightarrow \nu_1 \frac{\sum\limits_{a \in nI} x_a}{r} &= - (Xa + Y) - \frac{1}{nr} \sum_{a \in nI} \left(C - 2\sum_{\substack{c \ge k+l+1 \\ a \sim c}} x_c\right) \\
        \Rightarrow \nu_1 \frac{ \sum\limits_{a \in nI} x_a}{r} &= - (Xa + Y + Zc) + 2 \frac{1}{r} \sum_{a \in nI} \left(\frac{1}{n} \sum_{\substack{c \ge k+l+1 \\ a \sim c}} x_c\right).
    \end{align*}
    Note that $\nu_1 \to -\frac{\sqrt{2}}{2}$, $Xa + Y + Zc \to \frac{\sqrt{2}}{2}$, and $\frac{\sum_{a \in nI} x_a}{r} \to 1$ since $nI$ is a constant positive proportion of the bottom. We also have $\frac{r}{n|I|} \to 1$, hence we then take the limit of the above equality to obtain
    \[\lim_{n \to \infty} \frac{1}{n|I|} \sum_{a \in nI} \left(\frac{1}{n} \sum_{\substack{c \ge k+l+1 \\ a \sim c}} x_c\right) \to 0. \tag{$\star$}\]
    This represents the fact that for bottom $a$, the average value of its neighbours in the back ought to converge to 0.
    
    Suppose FTSOC that there exists $\varepsilon, \delta > 0$ such that there are arbitrarily large $n$ with $a' \in n(\frac14+\delta, \frac12]$ where $d_{a'} > (\frac12 - \frac{\sqrt{2}}{4} + \varepsilon)n$. Since the $d_a$ are decreasing, we have for all $a \in n(\frac14, \frac14+\delta]$ that $d_a > (\frac12 - \frac{\sqrt{2}}{4} + \varepsilon)n$. Yet for each of these, 
    \[\sum\limits_{\substack{c \ge k+l+1\\a\sim c}} x_c \ge \sum\limits_{c \ge (\frac12 + \frac{\sqrt{2}}{4} - \varepsilon)n} x_c.\]
    
    By an analogous argument to above for smoothing(a), we have that the average of the last $\frac{t}{k}$ vertices in the back converges to $\sqrt{2}-1$. Given that $\frac{t}{k} \to \frac{n}{4}$ as $n \to \infty$, we then deduce that
    \[\frac{\sum_{c \ge \frac34 n} x_c}{\frac14 n} \to \sqrt{2}-1.\]
    since $0 \le x_c \le 1$ bounded. Additionally, $x_c \le 1$ yields 
    \[\sum\limits_{\frac34 n \le c \le (\frac12 + \frac{\sqrt{2}}{4} - \varepsilon)n} x_c \le \left(\frac{\sqrt{2}-1}{4} - \varepsilon\right) n.\]
    
    Hence by convergence, for all $\varepsilon' > 0$, we have for $n$ sufficiently large, 
    \[\sum\limits_{c \ge (\frac12 + \frac{\sqrt{2}}{4} - \varepsilon)n} x_c \ge (\varepsilon - \varepsilon')n.\] 
    
    But then for $I = (\frac14, \frac14 + \delta)$ and $n$ sufficiently large,
    \[\frac{1}{n|I|} \sum_{a \in kI} \left(\frac{1}{n} \sum_{\substack{c \ge k+l+1 \\ a \sim c}} x_c\right) \ge \varepsilon - \varepsilon' > 0\]
    which is a contradiction of $(\star)$.
\end{proof}

\begin{clm}\label{clm:small-box}
    There exists a small constant $\eta$ such that for sufficiently large $n$, there is always an edge between $x_a, x_c$ for $a \in n(\frac14, \frac14 + \eta)$ and $c \in n(1-\eta, 1)$ in $M_n$. 
\end{clm}
\begin{proof}
    Suppose FTSOC there isn't. Note that for sufficiently large $n$, by Claim \ref{clm:degree} the number of edges between $a \in n(0, \frac14)$ and $c$ is at most $\frac14n \cdot \frac14n + \frac14n \cdot (\frac12 - \frac{\sqrt{2}}{4} + \varepsilon)n < (\frac18 - \delta)n^2$ for some fixed positive $\delta$ as $\varepsilon \to 0$. Thus, since we have a converging to 0 proportion of edges between $a \in n(1/4, 1/2)$ and $c \in n(1/2, 3/4)$, we must have at least $\delta n^2$ edges between $a \in n(1/4, 1/2)$ and $c \in n(3/4, 1)$. 
    
    Since $d_a$ is decreasing, each $x_a$ with $a > (\frac14 + \eta)n$ contributes at most $\eta n$ edges, and every $a \in n(\frac14, \frac14 + \eta)$ contributes at most $n/2$. But then we get $\delta n^2 \le \eta(\frac14 - \eta) n^2 + \frac{\eta}{2} n^2$ which gives $\delta \le \eta(\frac14 - \eta) + \frac{\eta}{2}$. But for $\eta$ sufficiently small, this is less than $\delta$.
\end{proof}

With these two claims and their symmetric counter-parts for $c$, we are ready to prove Theorem \ref{thm:n/2-regular}.

\begin{proof}[Proof of Theorem \ref{thm:n/2-regular}]
    We re-consider the derivation of extrema(c) for the Type 2 eigenvector. Similar to before, if there exists a constant positive proportion $\delta$ of the edges where $\frac{x_a}{x_1} < 1-\epsilon$, then we get $\mu_2(c-1) \le 2\frac{l}{(n-k-l)B} (x_1(t - \delta n^2) + x_1(1 - \epsilon)\delta n^2)$ hence $\nu_2^2(c-1) \ge 2\frac{T - \epsilon \delta}{Z}(\nu + 2Zc)$. This also gives $\nu_1 + \nu_2$ strictly greater than $-\sqrt{2}$ by some $\gamma(\epsilon, \delta) > 0$ since again the equality state is infeasible. Identically this also holds if $B = 0$. Thus, for sufficiently large $n$, we must have that if some constant proportion sub-interval of $\{x_a\}$ induces a positive proportion of edges, their average value converges to $x_1$. 

    By Claim \ref{clm:small-box} we have that for sufficiently large $n$ and some fixed $\eta$, $a \in n(\frac14, \frac14 + \eta)$ induce a positive proportion of edges. Thus, the average of $\frac{x_a}{x_1}$ over any non-trivial sub-interval of $n(0, \frac14 + \eta)$ converges to 1. Yet $x_1$ has all the back vertices as neighbours whereas for $a$ near $(\frac14 + \eta)n$, $d_a \le (\frac12 - \frac{\sqrt{2}}{4})n$ in the limit. Thus, the average of $\frac{x_c}{x_1}$ for $c \in n(\frac12, \frac12 + \frac{\sqrt{2}}{4})$ converges to 0. 
    
    However, we also have by the analogous version of Claim \ref{clm:small-box} for $c$, that for sufficiently large $n$ and $c \in n(\frac34 - \eta, \frac34)$, $d_c \ge \eta n$ where all edges satisfy $a \in n(0, \frac14 + \eta)$. Thus, for these $c$ we have $\mu_2 x_c \ge 2\sum_{a \le \eta n} x_a - A - B - C$ which implies $\nu_2 x_c \ge \frac{2}{n} \sum_{a \le \eta n} x_a + \nu_2 x_b$. Thus, we have $\frac{x_c}{x_1} \le \frac{2\eta}{\nu_2} + \frac{x_b}{x_1}$. With $c \to -\infty$, the last term converges to 0, and with the rest we get $\frac{x_c}{x_1} \le -2\sqrt{2} \eta$, contradicting convergence to 0. 

    Thus, we cannot have $\nu_1(M_n) + \nu_2(M_n) \to -\sqrt{2}$. In particular, there must exist $\varepsilon > 0$ such that $\nu_1(M) + \nu_2(M) > -\sqrt{2} + \varepsilon$ for all such $M = 2Q-J$. 
\end{proof}

We remark that whilst the equality cases for Type 1 and Type 2 when $\nu_1 + \nu_2 = -\sqrt{2}$ cannot simultaneously occur, they can occur separately. Indeed, for Type 1 if we have an edge iff $a \le (\frac12 - \frac{\sqrt{2}}{4})n$ or $c \ge (\frac12 + \frac{\sqrt{2}}{4})n$, then we get $\nu_1 = -\frac{\sqrt{2}}{2}$ and $\nu_2 \approx -0.6245$. Similarly, for Type 2 add an edge iff $a \le \frac{n}{2\sqrt{2}}$ and $c \ge n - \frac{n}{2\sqrt{2}}$, then we get $\nu_1 \approx -0.6245$ and $\nu_2 = -\frac{\sqrt{2}}{2}$. This parallel in spectra is explained by them deriving from the same $n/2$-regular graph, just taken at different rotations.

\begin{figure}[H]
    \centering
    \begin{tabular}{cc}
        \includegraphics[width = 0.25\textwidth]{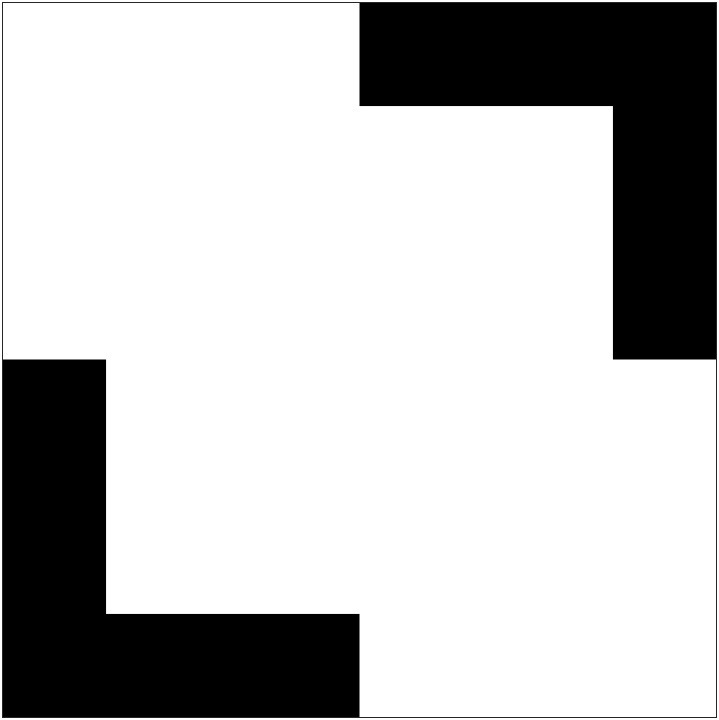} & \includegraphics[width = 0.25\textwidth]{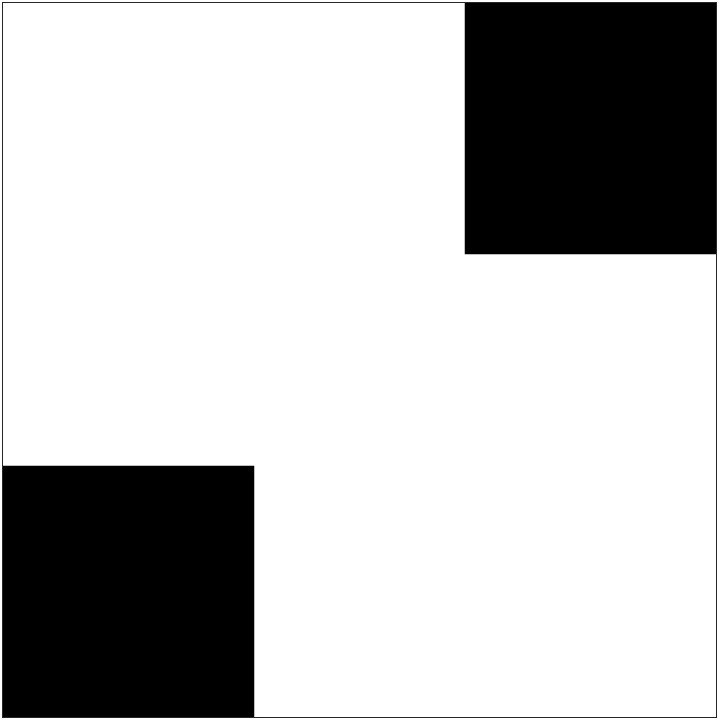}
    \end{tabular}
    \caption{Type 1 (left) and Type 2 (right) equality cases for $M$.}
    \label{fig:equality}
\end{figure}

\section{Proof of Theorem \ref{thm:worst-sum}}\label{section:proof}

We disprove the worst-case scenario by showing it is close to being an $n/2$-regular invariant graph in every aspect. 

Suppose FTSOC we have a sequence of graphs $G_n$ where $\frac{\lambda_{n-1}(G_n) + \lambda_n(G_n)}{|V(G_n)|} \to -\frac{\sqrt{2}}{2}$ as $n \to \infty$. By appealing to the techniques of Nikiforov \cite{nikiforov2006combinations}, i.e. by taking closed blow-ups and adding $o(n)$ isolated vertices if necessary, assume that $|V(G_n)| = n$. Thus, we have $\lambda_{n-1}(G_n) + \lambda_n(G_n) = -\frac{\sqrt{2}}{2}n + o(n)$. 

\begin{clm}\label{clm:wlog-invariant}
    WLOG we have $G_n = G_n^*$ for every $n$.
\end{clm}
\begin{proof}
    Consider $G$ given by Corollary \ref{cor:worst-invariant}. By Nikiforov \cite{nikiforov2015extrema}, $\lambda_{n-1}(G_n) + \lambda_n(G_n) = -\frac{\sqrt{2}}{2}n + o(n)$ is the equality case for the upper bound, hence we must also have $\lambda_{n-1}(G) + \lambda_n(G) = -\frac{\sqrt{2}}{2}n + o(n)$ for a smaller proportion of $\lambda_{n-1} + \lambda_n$. Thus, we can WLOG replace $G_n$ with $G$. 
\end{proof}

Given invariance, we now prove that $G_n$ is close to having the structure of an $n/2$-regular invariant graph, analogous to Theorem \ref{thm:invariant-symmetry}. Implicitly assume $G = G_n$ with $n$ vertices.

\begin{clm}\label{clm:close-symmetry}
    Given any rotation of the labelling in Theorem \ref{thm:operation-run}, we can add/remove $o(n^2)$ edges from $G$ to obtain a $n/2$-regular $K$ where
    \[A(K) = \begin{pmatrix}
        Q & J_{n/2} - Q \\
        J_{n/2} - Q & Q
    \end{pmatrix},\]
    the complement of $Q$ is in a perfect elimination ordering, $Q$ is bipartite between the front and back with a non-empty middle, and $\lambda_{n-1}(K) + \lambda_n(K) = -\frac{\sqrt{2}}{2} n + o(n)$. 
\end{clm}
\begin{proof}
    We have $\lambda_{n-1}^2 + \lambda_n^2 \ge \frac{n^2}{4} + o(n^2)$ by AM-QM, whence $2e(G) - \lambda_1^2 \ge \frac{n^2}4 + o(n^2)$. But $\lambda_1 \ge \frac{2e(G)}{n}$, hence by the equality case of AM-GM, $\lambda_1 = \frac{n}{2} + o(n)$ with additionally $e(G) = \frac{n^2}{4} + o(n^2)$. Now $\lambda_1 - \frac{2e(G)}{n} = o(n)$ hence by the bounds of Nikiforov \cite{nikiforov2007nordhaus}, $s(G) = \sum_{i = 1}^n |d(i) - \frac{2e(G)}{n}| = o(n^2)$. In particular, $\frac{2e(G)}{n} = \frac{n}{2} + o(n)$ hence $\sum_{i = 1}^n |d(i) - \frac{n}{2}| = o(n^2) + n \cdot o(n) = o(n^2)$. 

    Suppose the top-left $\frac{n}{2} \times \frac{n}{2}$ sub-matrix of $A(G)$ is $P$. Since $G = G^*$, it has the structure in Theorem \ref{thm:operation-run}. Thus we can argue identically to Theorem \ref{thm:invariant-symmetry} to show $P$ already satisfies the desired properties of $Q$, up to the possibility of front vertices having a run contained strictly within $\{1, \dots, n/2\}$ (and symmetrically, back vertices having a run contained strictly within $\{n/2+1, \dots, n\}$). However, these result in $d(i) < n/2$ and hence take up at most $o(n^2)$ non-edges in the top-right and bottom-left corners since $\sum_{i = 1}^{n/2} |d(i) - \frac{n}{2}| = o(n^2)$. We can add in these $o(n^2)$ edges to $P$, and can also empty out a row in between the front and back with $O(n) = o(n^2)$ edge removals to produce the desired $Q$. Furthermore, replacing the top-left sub-matrix of $G$ with $Q$ is $o(n^2)$ edge additions/removals, and by Weyl's inequality, only affects the spectrum up to $o(n)$. We also still have $\sum_{i = 1}^{n} |d(i) - \frac{n}{2}| = o(n^2)$.

    Now consider replacing the top-right and bottom-left sub-matrices with $J_{n/2} - Q$. Since $\sum_{i = 1}^{n/2} |d(i) - \frac{n}{2}| = o(n^2)$ and the neighbourhoods of each vertex were previously also consecutive runs (which other than $o(n^2)$ non-edges, they go over vertex $n/2$), this is fixing every $d(i) = n/2$ in the first $n/2$ vertices and hence involves only $o(n^2)$ edge additions/removals.

    Finally, note that we have made only $o(n^2)$ changes to the neighbourhoods of the last $n/2$ vertices. Thus, we still have $\sum_{i = n/2+1}^{n} |d(i) - \frac{n}{2}| = o(n^2)$ and hence by adjusting each run to complete a $n/2$-run, i.e. replacing the bottom-right sub-matrix with $Q$, this is only another $o(n^2)$ edge additions/removals. Our spectrum is the same up to $o(n)$, and in particular $\lambda_{n-1}(K) + \lambda_n(K) = -\frac{\sqrt{2}}{2} n + o(n)$. 
\end{proof}

We then prove the analogous result to Type 1 and Type 2 eigenvectors. 

\begin{clm}\label{clm:close-type}
    There exists a canonical rotation where the first $n/2$ components of $x, y$ give $x'$ of Type 2 and $y'$ of Type 1. 
\end{clm}
\begin{proof}
    We proceed similarly to the proof of Theorem \ref{thm:2Q-J}. Indeed, we instead have the more general $\lambda_{n-1} x_{i+1} - \lambda_{n-1} x_i = \sum \pm x_k$, but the $\pm x_k$ lie between the vectors orthogonal to $x_i, x_{i+1}$ by definition, hence so must $\sum \pm x_k$. Thus, vector coordinates are still monotone within a quadrant. 

    Consider the span of vectors with arguments $[0, \pi)$ and $[\pi, 2\pi)$. One of these contains at least $n/2$ vectors, WLOG $[0, \pi)$. The restriction of $x, y$ to the first $n/2$ vectors anti-clockwise from the positive $x$-axis again yields Type 2 $x'$ and Type 1 $y'$.
\end{proof}

This proves that $G$ is close enough to apply Theorem \ref{thm:n/2-regular} up to some controlled error, which is sufficient to conclude. 

\begin{clm}
    Let $x$ be an eigenvector of $G$ corresponding to $\lambda_{n-1}$ or $\lambda_n$. Then the equations in Stage 1 of Theorem \ref{thm:almost-n/2-regular} hold with $x'$ and $\mu = -\frac{1}{2\sqrt{2}}n + o(n)$. 
\end{clm}
\begin{proof}
    Let $\sim$ denote equal up to $o(n)$. Again by AM-QM and the bounds of Nikiforov \cite{nikiforov2015extrema}, $\lambda_{n-1} + \lambda_n \sim -\frac{\sqrt{2}}{2}n$ which gives $\lambda_{n-1}, \lambda_n \sim -\frac{1}{2\sqrt{2}}n$. Hence, by Claim \ref{clm:close-symmetry}, $G$ is $o(n^2)$ edges away from $K$ with $\lambda_{n-1}(K), \lambda_n(K) \sim -\frac{1}{2\sqrt{2}}n$. Furthermore, the last eigenvectors of $K$ look like $\binom{1}{-1} \otimes z$ for a $n/2$-vector $z$, as the spectrum of $J$ is entirely non-negative hence the last eigenvalues come from $2Q-J$. 

    Suppose FTSOC that there exists $\epsilon > 0$ such that there are arbitrarily large $n$ where the unit eigenvector $x$ has at least an $\epsilon$ component orthogonal to the span of the last two eigenvectors $\binom{1}{-1} \otimes z_{n-1}, \binom{1}{-1} \otimes z_n$ of $K$, i.e. $x = \alpha \binom{1}{-1} \otimes z_{n-1} + \beta \binom{1}{-1} \otimes z_n + \varepsilon w$ for $\{\binom{1}{-1} \otimes z_{n-1}, \binom{1}{-1} \otimes z_n, w\}$ orthonormal, $\alpha^2 + \beta^2 + \varepsilon^2 = 1$ and $\varepsilon \ge \epsilon$. Noting that $\lambda_{n-2} \ge -\frac{1}{2\sqrt{3}}n$, we have $\frac{\langle x, Kx\rangle}{\langle x, x\rangle} \ge (\alpha^2 + \beta^2) (-\frac{n}{2\sqrt{2}} + o(n)) + \epsilon^2 (-\frac{n}{2\sqrt{3}}) > (-\frac{1}{2\sqrt{2}} + \delta(\epsilon))n + o(n)$ for some $\delta(\epsilon) > 0$, contradiction. Thus, $x$ has an $o(1)$ component not of the form $\binom{1}{-1} \otimes z$. 

    In particular, choosing the canonical rotation from Claim \ref{clm:close-type},
    \[-\frac{n}{2\sqrt{2}} \sim \frac{\langle x, A(G)x \rangle}{\langle x, x\rangle} \sim \frac{\langle x, Kx \rangle}{\langle x, x\rangle} \sim \frac{\langle \binom{1}{-1} \otimes x', K \binom{1}{-1} \otimes x' \rangle}{\langle \binom{1}{-1} \otimes x', \binom{1}{-1} \otimes x'\rangle} \sim \frac{\langle x', (2Q-J)x'\rangle}{\langle x', x'\rangle}\]
    up to $o(n)$, and similarly for $y'$. 

    Note that the equations in Theorem \ref{thm:almost-n/2-regular} hold exactly for $z_{n-1}$ and $z_n$ with $\mu = -\frac{1}{2\sqrt{2}}n + o(n)$, hence since $x'$ is a linear combination up to a $o(1)$ difference and the equations are linear in eigenvector components, they also hold for $x'$ with $\mu = -\frac{1}{2\sqrt{2}}n + o(n)$. 
\end{proof}

\begin{proof}[Proof of Theorem \ref{thm:worst-sum}]
    After normalising, we get $\nu + o(1) = \frac{\mu + o(n)}{n}$, and thus the equations for $\nu$ are satisfied up to $o(1)$. The inequalities derived from these equations follow from the Types of $x', y'$ recovered in Claim \ref{clm:close-type}. Thus, by Theorem \ref{thm:n/2-regular} there exists an $\varepsilon > 0$ such that $\frac{\lambda_{n-1} + \lambda_n}{n} > -\frac{\sqrt{2}}{2} + \varepsilon$, contradiction. 
\end{proof}

\section{Concluding remarks}

Having established a strengthened upper bound on $c_3$ and in light of Theorem \ref{thm:block}, this makes it more plausible that we indeed have $c_3 = \frac13$. In fact, we have the following conjecture as a strengthened version of the original, which we've verified for all graphs of order at most 9. Though there are possible variations, we believe this is the most natural formulation.

\begin{cnj}\label{cnj:sum}
    Given a graph $G$ of order $n$, we have $\lambda_{n-1} + \lambda_n \ge -\frac{2n}{3}$. 
\end{cnj}

We were not able to make use of the additional structure in Section \ref{section:structure} for a general $G = G^*$. Our attempts revolved around reducing to the pivalous graph, with intuition from Theorem \ref{thm:operation-regular}. Yet, for example, it is not true that replacing two copies of the same row with one copy always decreases $\frac{\lambda_{n-1} + \lambda_n}{n}$, nor that $Pi_k^{[t]}$ is minimal over all $G = G^*$ of order $kt$ and $k$ distinct rows. \\

Our methods in Section \ref{section:n/2-regular} leave room for improvement in Theorem \ref{thm:n/2-regular} with the lack of interaction between Type 1 and Type 2 bounds, leading to the $\nu_1 = \nu_2 = -\frac{\sqrt{2}}{2}$ situation when realistically only one equality can hold at once. However, this cannot then be adapted to improve Theorem \ref{thm:worst} by virtue of $n/2$-regularity up to $o(n)$ being forced by the algebraically worst case scenario. Furthermore, this also means that our proof doesn't provide any actual bounds on $c_3$ and instead eliminates the $c_3 = \frac{1}{2\sqrt{2}}$ case. \\

In addition, we comment on alternative proof directions for Theorem \ref{thm:almost-n/2-regular} and outline possible sketches, with analytic ideas. Firstly, it seems plausible that one can reduce the structure above to ensure feasibility but leave just enough for only one structure to be permissible. Then the sequence of worst-case graphs converges in cut metric, and as demonstrated by Borgs et al. \cite{borgs2008graphon}, the spectrum of the limit graphon is $\frac12, \frac{1}{2\sqrt{2}}, \frac{1}{2\sqrt{2}}, 0, 0, \dots$, but must have trace $\le 1$, contradiction. However, we found that demonstrating the structure cannot exist is simpler than fixing it down.

Alternatively, having deduced the structure of $Q$ in Theorem \ref{thm:invariant-symmetry}, one may also be able to adapt analytic techniques to provide intuition for the block construction yielding minimal $\lambda_{n-1} + \lambda_n$. This has been explored by Terpai \cite{terpai2011analytic}, Breen, Riasanovsky, Tait, Urschel \cite{breen2022}, Liu \cite{liu2024} and more recently by Kumar, Liu, Monterde, Pragada and Tait \cite{kumar2026}. 

Indeed, both operations of stretching and averaging components in an interval fix Type 1 and Type 2 properties. Then we can run the same arguments with instead $x_0 \in (0, X+Y)$ and an interval neighbourhood $R = (x_1, x_2)$ around $x_0$ where the $\mathcal{L}^2$ average of eigenfunction $f$ on $R$ is $f(x_0)$. Letting $\tilde{f}$ be $f$ except $\tilde{f}(x) = f(x_0)$ on $R$, we can similarly prove that the change in the integral is $\tilde{I} - I = -(l(R) f(x_0) - \int_R f(x) dx)^2 \le 0$. Since both Type 1 and Type 2 are monotone on $(0, X+Y)$, we have that $W$ must be a step with at most two steps on $(0, X+Y)$, and identically for $(X, 1)$. 

The primary roadblock here is the lack of orthogonality. In fact, we have that if $f_1(x_0)$ and $f_2(x_0)$ are both $\mathcal{L}^2$ averages on $R$, then by Cauchy-Schwarz, 
\begin{align*}
    f_1(x_0) f_2(x_0) &= \left(\frac{1}{l(R)} \int_R f_1(x)^2 dx\right)^{1/2} \left(\frac{1}{l(R)} \int_R f_2(x)^2 dx\right)^{1/2} \\
    &\ge \frac{1}{l(R)} \int_R f_1(x) f_2(x) dx \\
    \Rightarrow \int_R \tilde{f_1}(x) \tilde{f_2}(x) dx &\ge \int_R f_1(x) f_2(x) dx
\end{align*}
with equality iff $f_1, f_2$ were linearly dependent on $R$. Yet this cannot happen in the front where Type 1 and Type 2 are oppositely oriented. \\

Finally, we haven't yet explored into the generalised operation in Subsection \ref{sub:generalisation}, but we hope it inspires some further research into the unveiled structure. Indeed, the strikingly different growth rates of the naive bounds for $\omega$ and $\chi$ may be of interest. 

Note that with $k \ge 3$ dimensions, there is no natural selection of the vectors `in order', hence we envisage it will be difficult to pursue the same approach without the ability to restrict eigenvectors to controlled Types.

\section*{Acknowledgements}

This research was conducted shortly after the author was a student in the University of Cambridge Summer Research in Mathematics (SRIM) programme, during which the author worked with Giacomo Leonida under the supervision of Jan Petr. The author is grateful to Giacomo Leonida for the fantastic discussions during the programme, resulting in an initial joint paper \cite{leonida2024structure}, as well as to Jan Petr for the weekly helpful meet-ups, write-up advice and for introducing the problem. The author is also grateful to SmartSaas for funding their participation in the SRIM programme. The author would like to thank
the anonymous reviewers for a careful reading and helpful suggestions.

\bibliographystyle{plain}
\bibliography{citations.bib}

\section*{Appendix}
\renewcommand{\thesection}{A}
\setcounter{thm}{0}

We deal with the block construction mentioned in the proof of Theorem \ref{thm:almost-n/2-regular} and prove that $C_6^{[t]}$ is optimal over all closed vertex multiplications of $C_6$ by $[a, b, c, a, b, c]$. 

\begin{thm}\label{thm:block}
    Suppose $G$ is the closed vertex multiplication of $C_6$ by $[a, b, c, a, b, c]$. Then $\lambda_2 + \lambda_3 \le \frac{2n}{3} - 2$, with equality iff $a = b = c$. 
\end{thm}
\begin{proof}
    Add a self-loop at each vertex and take the divisor matrix for the equitable partition. 
    \[\begin{pmatrix}
        a & b & 0 & 0 & 0 & c \\
        a & b & c & 0 & 0 & 0 \\
        0 & b & c & a & 0 & 0 \\
        0 & 0 & c & a & b & 0 \\
        0 & 0 & 0 & a & b & c \\
        a & 0 & 0 & 0 & b & c
    \end{pmatrix}.\]
    
    This has characteristic polynomial $x^2(x-(a+b+c))P(x)$ where
    \[P(x) = x^3 - (a+b+c)x^2 + 4abc.\]
    
    Thus, the eigenvalues of $G$ are 0, $a+b+c$ which corresponds to the regularity, and the roots of $P(x) = 0$. Normalising $a + b + c = \frac12$, we get $x^3 - \frac12 x^2 + 4\alpha$ where $0 \le \alpha \le \frac1{216}$ due to AM-GM. Note that $P(1/3) = -\frac1{54} + 4\alpha \le 0$, hence it either lies between the first two roots or before the last root. Yet $P(0) = 4\alpha \ge 0$ hence the last root is non-positive. 

    If $\frac13 + t$ is a root for $P$ with $0 \le t \le \frac13$, then
    \begin{align*}
        0 &= \left(\frac{1}{27} + \frac13 t + t^2 + t^3\right) - \frac12 \left(\frac19 + \frac23 t + t^2\right) - 4\alpha \\ 
        \Rightarrow -\frac23 t - 2t^3 + \frac23 t &= \left(\frac{1}{27} - \frac13 t + t^2 - t^3\right) - \frac12 \left(\frac19 - \frac23 t + t^2\right) - 4\alpha \\
        \Rightarrow P\left(\frac13 - t\right) &= -2t^3 \le 0.
    \end{align*}
    
    Thus, $\frac13 - t$ lies between the first two roots. In particular, $\frac13 + t$ must be the first root, and the sum of the first two roots is at most $\frac23$, with equality iff $t = 0$ iff $\frac13$ is a root. This happens at $abc = \frac1{216}$ and hence $a = b = c = \frac16$ as the equality case for AM-GM. 

    Since we obtained the eigenvalues of $A(G) + I$ from $P$, we have $(\lambda_2 + 1) + (\lambda_3 + 1) \le \frac{2n}{3}$ with equality iff $a = b = c$. 
\end{proof}

\end{document}